\documentclass[a4paper,11pt]{amsart}%{article}

\usepackage[top=2.5cm, bottom=2.5cm, left=2.5cm, right=2.5cm]{geometry}

\usepackage{amssymb}
\usepackage{ascmac}
\usepackage[dvipdfmx]{graphicx}
\usepackage{color}
\usepackage{blkarray}
\usepackage{multirow,bigdelim}

\usepackage[colorlinks=true,bookmarks=true,
bookmarksnumbered=true,bookmarkstype=toc,linktocpage=true
]{hyperref}

\newtheorem{Def}{Definition}[section]
\newtheorem{Thm}[Def]{Theorem}
\newtheorem{Lem}[Def]{Lemma}
\newtheorem{Assumption}[Def]{Assumption}
\newtheorem{Rem}[Def]{Remark}
\newtheorem{Cor}[Def]{Corollary}
\newtheorem{Prop}[Def]{Proposition}

\numberwithin{equation}{section}

        \newcommand{\Tr}{\operatornamewithlimits {Trace}}

\allowdisplaybreaks

\usepackage{ascmac}

\newcommand{\mcf}{\mathcal{F}}

\newcommand{\mcl}{\mathcal{L}}

\newcommand{\mcm}{\mathcal{M}}

\newcommand{\mcn}{\mathcal{N}}

\newcommand{\mbbu}{\mathbb{U}}

\newcommand{\mbbh}{\mathbb{H}}
\newcommand{\bmbbh}{\bar{\mathbb{H}}}
\newcommand{\mbbn}{\mathbb{N}}

\newcommand{\mbbr}{\mathbb{R}}

\newcommand{\mbfx}{\mathbf{X}}
\newcommand{\mbfz}{\mathbf{Z}}

\newcommand{\al}{\alpha}
\newcommand{\del}{\delta}
\newcommand{\ep}{\epsilon} 
\newcommand{\vp}{\varphi}

\newcommand{\D}{\Delta}
\newcommand{\sig}{\sigma}

\newcommand{\lam}{\lambda}

\newcommand{\gam}{\gamma}
\newcommand{\Gam}{\Gamma}
\newcommand{\p}{\partial}

\newcommand{\Ez}{E_{\mbfz_n}}
\newcommand{\Xk}{X_{t_k}}
\newcommand{\pXk}{X_{t_{k-1}}}
\newcommand{\Zk}{Z_{t_k}}
\newcommand{\pZk}{Z_{t_{k-1}}}

\newcommand{\cil}{\xrightarrow{\mcl}} % <- Convergence in law
\newcommand{\cip}{\xrightarrow{p}} % <- Convergence in probability

\newcommand{\diag}{\mathop{\rm diag}} 
\newcommand{\kl}{\text{KL}}
\newcommand{\ic}{\text{IC}}
\newcommand{\el}{\text{EL}}

\def\nn{\nonumber}

\def\suml{\sum_{l_1=1}^{d}}
\def\sumk{\sum_{k=1}^n}

\def\intk{\int_{t_{k-1}}^{t_k}}

\def\intd{\int_{\mbbr^d}}
\def\ezn{E_{\mbfz_n}}

\def\dim{\mathrm{dim}}

\def\tz{\theta_{0}}
\def\az{\alpha_{0}}
\def\lz{\lambda_{0}}
\def\tes{\hat{\theta}_{n}}
\def\les{\hat{\lambda}_{n}}
\def\aes{\hat{\alpha}_{n}}

\title[]
{Predictive information criterion for jump diffusion processes}
\date{\today}

\author{Yuma Uehara}
\address[Department of Mathematics, Faculty of Engineering Science, Kansai University, 3-3-35 Yamate-cho, Suita-shi, Osaka 564-8680, Japan
]{}
\email{y-uehara@kansai-u.ac.jp}

\begin{document}

\maketitle

\begin{abstract}
In this paper, we address a model selection problem for ergodic jump diffusion processes based on high-frequency samples. We evaluate the expected genuine log-likelihood function and derive an Akaike-type information criterion based on the threshold-based quasi-likelihood function. In the derivation process, we also give new estimates of the transition density of jump diffusion processes.
We also provide the relative selection probability of the proposed information criterion.
\end{abstract}

\section{Introduction}

In this paper, we suppose that we observe an equally spaced high-frequency sample $\mbfx_n=\{X_{kh_n}\}_{k=1}^n$ with a positive decreasing sequence $\{h_n\}$ satisfying $T_n=nh_n\to\infty$ from the following $d$-dimensional ergodic jump diffusion process defined on a stochastic basis $(\Omega, \mcf,P)$:
\begin{align*}
        &dX_t=A(X_t)dt+B(X_t)dw_t+\int zN(dt,dz),
\end{align*}
where $w$ is a $d$-dimensional standard Wiener process, the coefficients $A:\mbbr^d\to\mbbr^d$ and $B:\mbbr^d\to\mbbr^d\times\mbbr^d$ are bounded measurable functions, and $N$ is a Poisson random measure whose mean measure is written as $f(z)dzdt$ with $\int f(z)dz<\infty$.
Moreover, $f(z)$ is supposed to be expressed as $\lambda_0 F(z)$ where $\lambda_0$ is a positive constant and $F(z)$ is a probability density function whose tail is bounded by a constant multiple of a Laplace probability density.
For such a setting, our objective is to construct a predictive information criterion to select a plausible one among candidate models:
\begin{align*}
        &dX_t=a_{m_1}(X_t,\theta_{m_1})dt+b_{m_2}(X_t,\sig_{m_2})dw_t+\int zN_{\lambda,\theta_{m_1}}(dt,dz),
\end{align*}
where $\theta_{m_1}\in\Theta_{\theta_{m_1}}\subset \mbbr^{p_{\theta_{m1}}}$, $\sig_{m_2}\in\Theta_{\sig_{m_2}}\subset \mbbr^{p_{\sig_{m_2}}}$ and $\lam\in(\underline{\lam},\bar{\lam})$ with $0<\underline{\lam}<\bar{\lam}<\infty$, $m_1\in\{1,\dots, M_1\}$ and $m_2\in\{1,\dots, M_2\}$.
We assume the nested condition; that is, each candidate model contains the true model.
Each coefficient is supposed to be bounded and satisfy some regularity condition as well, and $N_{\lambda,\theta_{m_1}}$ is a Poisson random measure and its mean measure is expressed as $f_{\theta_{m_1},\lam}=\lam F_{\theta_{m_1}}$ with probability density $F_{\theta_{m_1}}$. We note that the parameter $\lam$ does not depend on $\theta$, and that the selection targets are the drift and diffusion coefficients, and the jump distribution.
There are several studies on model selection for stochastic processes by information criteria: mixing processes (\cite{UchYos01}), diffusion processes (\cite{Uch10} and \cite{EguMas18}), multivariate CARMA processes \cite{FasKim17}, L\'{e}vy driven SDE models (\cite{EguUeh21} and \cite{EguMas24}).

Let $\mbfz_n$ be an independent copy of $\mbfx_n$ and we write $t_k=kh_n$.
Based on the time homogeneity of $X$, we write its transition density function as $q_t(x,y;\al,\lam)$. 
We note that the functional form of $q_t(x,y;\al,\lam)$ is not obtained explicitly in general.
Combined with the Markov property of $X$, the log-likelihood function $l_n$ can be expressed as 
\begin{equation*}
  l_n(\mbfx_n, \theta)=\sumk \log q_{h_n}(\pXk, \Xk;\al,\lam).
\end{equation*}
To conduct Akaike information criterion (AIC, \cite{Aka74}) based model selection, we will evaluate
\begin{equation*}
    E_{\mbfz_n}[l_n(\mbfz_n,\aes,\les)]=E_{\mbfz_n}\left[\sumk \log q_{h_n}(\pZk, \Zk;\aes,\les)\right],
\end{equation*}
where $E_{\mbfz_n}$ and $(\aes,\les)$ denote the expectation operator with respect to $\mcl(\mbfz_n)$ and the threshold-based quasi-maximum likelihood estimator constructed by $\mbfx_n$, respectively. 
In the case where the underlying model is diffusion processes, \cite{Uch10} utilizes the theory of Malliavin calculous and stochastic flow in order to evaluate the (derivative of) transition density function and the paper gives an asymptotically unbiased estimator of $E_{\mbfz_n}[l_n(\mbfz_n,\aes,\les)]$ building on the Gaussian quasi-likelihood function.
This approach also works for jump diffusion processes when we evaluate the derivative of the transition density with respect to the diffusion parameter and drift parameter which is not contained in the jump component. 
Indeed \cite{KohNuaTra17} and \cite{TraNgo23} derive the local asymptotic (mixed) normality of jump diffusion models based on this approach.
However, when the jump structure is additionally parametrized as in our case, the Poisson random measure may possibly depend on the parameter in some non-linear way, and it makes difficult to evaluate the derivative of the transition density with respect to the jump parameter in a similar manner.
For this problem, we consider the expansion of the transition density as in \cite{OgiUeh23} which showed the local asymptotic normality for jump diffusion processes whose jump component is parametrized, and we obtain an asymptotically unbiased estimator of $E_{\mbfz_n}[l_n(\mbfz_n,\aes,\les)]$ which is based on threshold-based quasi-likelihood function introduced in the next section.

The rest of this paper is constructed as follows. In Section \ref{Main}, we provide our technical assumptions and the main result. In Section \ref{Numerical}, the simulation result of model selection based on our information criterion is given. All of our technical proofs are presented in Section \ref{Appendix}.

\section{Main result}\label{Main}

To avoid notational confusion, we consider the following single model:
\begin{align*}
        &dX_t^{\al,\lam}=a(X_t^{\al,\lam},\theta)dt+b(X_t^{\al,\lam},\sig)dw_t+\int zN_{\lam,\theta}(dt,dz),
\end{align*}
where $\al=(\theta,\sig)$ and the mean measure of the Poisson random measure $N_{\lam,\theta}$ is expressed as $\lam F_\theta$ with a probability density $F_\theta$.
Similarly to the introduction, we also write the parameter space of the unknown $p_\theta$ dimensional parameter $\theta$ and the $p_\sig$ dimensional parameter $\sig$ as $\Theta_\theta$ and $\Theta_\sig$, respectively.
The true value is written as $\lam_0$ and $\al_0$, and we often abbreviate $X^{\al,\lam}$ as $X$ below.

Before we provide our technical assumptions, we introduce the notation used in the rest of this paper.
\begin{itemize}
\item We write $S(x,\sig)=(b(x,\sig))^{\otimes2}$. 
\item We write $\D_k X= X_{t_k}-X_{t_{k-1}}, \ \bar{X}_k(\theta)=\D_k X-h_na(X_{t_k},\theta)$.
% \item For any differentiable function $f$, $\p_{l_1} f$ denotes the partial derivative of $f$ with respect to $l_1$-th variable.
\item $I_d$ denotes the $d\times d$-identity matrix.
\item For a matrix $A$, $\Tr(A)$ stands for its trace.
\item For a vector $x=(x_1,\dots,x_k)^\top$, we write $$\p_x=\bigg(\frac{\p }{\p x_1},\dots,\frac{\p}{\p x_k}\bigg)^\top$$ where $\top$ is the transpose operator.
\item For a vector $y$, $|y|$ and $||y||_1$ represent its Euclidean norm and $L_1$-norm, respectively.
\item For positive sequences $\{a_n\}$ and $\{b_n\}$, $a_n\lesssim b_n$ means that 
$$\limsup_{n\to\infty} \frac{a_n}{b_n}<\infty.$$
\item For a function $f$, $\text{supp}(f)$ denotes its support and for a set $A$, $\p A$ stands for its boundary.
\item $\cil$ and $\cip$ represent convergence in distribution and convergence in probability, respectively.
\item $C$ denotes a generic positive constant that possibly varies from line to line.
\end{itemize}

\subsection{Assumptions}

Let $\eta$ be the law of the initial variable $X_0$.
\begin{Assumption}\label{inidist}(Initial distribution)
    For any $q>0$, 
    $$\int_{\mbbr^d}|x|^q\eta(dx)<\infty.$$
\end{Assumption}

\begin{Assumption}\label{pspace}(Parameter space)
The parameter spaces $\Theta_\theta$ and $\Theta_\sig$ are bounded convex domains in $\mbbr^{p_\theta}$ and $\mbbr^{p_\sig}$, respectively.
Moreover, there exist positive constants $\underline{\lam}$ and $\bar{\lam}$ such that $\lz\in(\underline{\lam}, \bar{\lam})$.
\end{Assumption}

\begin{Assumption}\label{coef}(Coefficients)
For all $i,j\in\{0,1,2,3,4,5\}$ satisfying $i+j\leq 5$, the derivatives $\p_x^i\p_\theta^j a(x,\theta)$ and $\p_x^i\p_\sig^j b(x,\sig)$ exist and are continuous on $\mbbr^d\times\Theta_\theta$ and $\mbbr^d\times\Theta_\sig$, respectively. In addition there exist positive constants $C$ and $\kappa$ such that
\begin{align*}
    & x^\top a(x,\theta)\leq -C |x|, \quad \text{for every $|x|$ large enough},\\
    &|a(x,\theta)|+|\p_x a(x,\theta)|+|b(x,\sig)|+|\p_x b(x,\sig)|\leq C,\\
    &|\p_x^i\p_\theta^j a(x,\theta)|+|\p_x^i\p_\sig^j b(x,\sig)|\leq C(1+|x|^\kappa),
\end{align*}
for all $i,j\in\{0,1,2,3,4,5\}$ satisfying $i+j\leq 5$ and $(\theta,\sigma)$.
\end{Assumption}

\begin{Assumption}\label{nondege}(Nondegeneracy)
    For any $x$ and $\sig$, $b(x,\sig)$ is symmetric and positive definite, and there exists a positive constant $C\geq1$ such that
    $$C^{-1}I_d\leq b(x,\sigma)\leq CI_d.$$
\end{Assumption}

\begin{Assumption}\label{jdist}(Jump distribution)
\begin{enumerate}
\item The support of $F_\theta$ does not depend on $\theta$, and for a constant
   $$\xi(\theta)=\sup_\beta\Bigg\{\beta\in\mbbr\bigg|\intd \frac{F_{\tz}(z)}{(F_\theta(z))^\beta}1_{F_{\tz}\neq0}(z)dz<\infty\Bigg\}\wedge1,$$
$\xi:=\inf_{\theta\in\Theta_\theta} \xi(\theta)>0.$
\item There exists a positive constant $C$ such that
$$F_\theta(z)\leq C e^{-C|z|},$$
for all $z\in\mbbr^d$ and $\theta\in\Theta_\theta$.
\item  The derivative $\partial^{k}_z\partial_\theta^l F_\theta(z)$ exists and is continuous with respect to $\theta\in \Theta_\theta$ for any $k\in\{0,1\}$, $l\in\{0,1,2,3,4,5\}$ and $z\in \mbbr^d$. Moreover, there exist constants $\gamma\geq 0$,  $C>0$, and $\epsilon'>0$ such that
\begin{align}
    \nn &|F_\theta(z)|1_{\{|z|\leq  \epsilon'\}}\leq C|z|^\gamma,
 \quad |\partial_\theta^l\log F_\theta(z)|1_{\{F_\theta(z)\neq0\}}\leq C(1+|z|)^{C},\\
\nn &|\partial_\theta^l\log F_\theta(z_1)-\partial_\theta^l\log F_\theta(z_2)|1_{\{F_\theta(z_1)F_\theta(z_2)\neq0\}}\leq C|z_1-z_2|(1+|z_1|+|z_2|)^{C}
\end{align}
for any $z,z_1,z_2\in \mbbr^d$, $\theta\in \Theta_\theta$, and $l\in \{1,2,3,4,5\}$.
\end{enumerate}
\end{Assumption}

Next, we introduce the balance condition on the sampling grid $\{h_n\}$.
\begin{Assumption}\label{balance-cond}(Sampling)
\begin{enumerate}
    \item $n^{-3/5}\lesssim h_n\lesssim n^{-4/7}.$
    \item There exists a positive constant $\eta>0$ such that
\begin{equation}
     nh_n^\frac{1+((d+\gamma)\rho)\wedge 1}{1+\eta}\to 0.
\end{equation}
as $n\to\infty$.
\end{enumerate}
\end{Assumption}

Due to the effect of the continuous part of $X$, the increment $\D_k X$ may not be included in $\text{supp}(F_\theta)$.
The next assumption is to handle such a probability, and the function $\varphi$ introduced in the following is for the truncation which is crucial in the construction of a quasi-likelihood function.

\begin{Assumption}\label{jsupp}(Truncation)
\begin{enumerate}
    \item There exist constants $\ep>0$ and $N_0\in\mbbn$ fulfilling that 
\begin{equation}
    \nn\int_{\{z: d(z, \partial\text{supp}(F_\theta))\leq h_n^\rho\}} F_\theta(z)dz \leq h_n^\ep
\end{equation}
for all $n\geq N_0$.
    \item There exists a sequence of real-valued differentiable functions $\{\varphi_n(z)\}_{n\in\mbbn}$ on $\mbbr^d$ satisfying the following properties: for all $n\in\mbbn$, $0\leq \varphi_n(z)\leq1$ and there exists a positive constant $M$ such that $\varphi(z)=0$ whenever $z\in D_n$ where
    \begin{align*}
        D_n=\bigcup_{k=0}^4\Bigg\{z\in\mbbr^d\bigg| \sup_\theta\Big|\p^k_\theta \log F_\theta(z)\Bigg|\geq \frac{M}{\ep^{k\vee1}_n}\Bigg\}\bigcup\bigcup_{k=0}^4 \Bigg\{z\in\mbbr^d\bigg|\sup_\theta\Big|\p_z\p_\theta^k\log F_\theta(z)\Big|\geq \frac{M(1+|z|^M)}{\ep^{k\vee1}_n}\Bigg\}.
    \end{align*}
    Furthermore, for all $n\in\mbbn$, the derivative $\p_z\varphi_n$ is continuous and satisfies
    $$\p_z\varphi_n(z)=0 \quad \text{on} \ D_n, \quad \text{and} \ \sup_z|\p_z\varphi_n(z)|=O(\ep_n^{-1}).$$
    \item There exist positive constants $C_5$ and $C_6$ such that
    \begin{align*}
        &\sup_\theta\int_{\mbbr^d}\bigg|\p_\theta^k \log F_\theta(z)\bigg|F_{\tz}(z)(1-\vp_n(z))1_{F_{\tz}\neq0}(z)dz\leq C_5h_n^{C_6}, \quad k=0,2,\\
        % &\sup_\theta\int_{\mbbr^d} |\p_\theta^l F_\theta(z)|(1-\varphi_n(z))dz\leq C_5h_n^{C_6}, \quad l=0,1,2.
        & \sup_\theta\int_{\mbbr^d}\big|F_\theta(z)-F_{\tz}(z)\big|(1-\vp_n(z))dz\leq C_5h^{C_6},\\
         &\int_{\mbbr^d} \big|\p_\theta^{2}F_{\tz}(z)\big|(1-\varphi_n(z))dz\leq C_5h_n^{C_6}.
    \end{align*}
    % and moreover, either of the following conditions holds: 
    % \begin{enumerate}
    %     \item  
    %     $\sup_\theta\int_{\mbbr^d} |\p_\theta^l F_\theta(z)|(1-\varphi_n(z))dz\leq C_5h_n^{C_6}, \quad l=0,1,2.$
    %     \item $\sup_\theta\int_{\mbbr^d} (F_\theta(z)-F_{\tz}(z))(1-\varphi_n(z))dz\leq C_5h_n^{C_{d,\gam,\rho,\eta}}$ with $\sqrt{nh_n}h_n^{C_{d,\gam,\rho,\eta}}\to0$
    % \end{enumerate}
\end{enumerate}
\end{Assumption}

Let $\theta^\star=(\theta,\lam)$.
We introduce the $(p_\theta+p_\sig+1)\times (p_\theta+p_\sig+1)$-dimensional matrix $\Gam=\diag(\Gam_\sig,\Gam_{\theta^\star})$ defined by
\begin{align*}
    [\Gamma_\sig]_{ij}&=\frac{1}{2}\int {\rm tr}(\partial_{\sigma_i}SS^{-1}\partial_{\sigma_j}SS^{-1})(x,\sigma_0)d\pi_0(x), \\
    [\Gamma_{\theta^\star}]_{ij}&=\int (\partial_{\theta_i^\star} a)^\top S^{-1}(\partial_{\theta_j^\star} a)(x,\alpha_0)d\pi_0(x)
    + \int \frac{\partial_{\theta_i^\star} f_{\theta_0}\partial_{\theta_j^\star} f_{\theta_0}}{f_{\theta_0}} 1_{\{f_{\theta_0}\neq 0\}}(z)dz.
\end{align*}

\begin{Assumption}
    $\Gam$ is positive definite.
\end{Assumption}

Under Assumptions \ref{inidist}, \ref{coef}, \ref{nondege}, and \ref{jdist}, \cite[Theorem 1.4]{Mas08} shows that the transition semigroup of $X=X^{\al_0,\lam_0}$ admits a unique invariant measure $\pi_0$.
For this measure, we define two functions $Y^1$ and $Y^2$ by
\begin{align*}
    &Y^1(\sig;\sig_0)=\frac{1}{2}\int \Tr\big(I_d-S^{-1}(x,\sig)S(x,\sig_0)\big)\pi_0(dx)-\frac{1}{2}\int\log\frac{\det S(x,\sig)}{\det S(x,\sig_0)}\pi_0(dx),\\
    &Y^2(\theta;\theta_0)=-\frac{1}{2}\int S^{-1}(x,\sig_0)\Big[\big(a(x,\theta)-a(x,\tz)\big)^{\otimes2}\Big]\pi_0(dx)+\int \Big(\log F_\theta(z)-\log F_{\tz}(z)\Big) F_{\tz}(z)dz.
\end{align*}
For these functions, we assume the following identifiability condition. 
\begin{Assumption}\label{indentifiability}(Identifiability)
    There exist positive constants $\chi_\sig$ and $\chi_\theta$ such that for all $\sig\in\Theta_\sig$,
    $$Y^1(\sig;\sig_0)\leq -\chi_\sig|\sig-\sig_0|^2,$$
    and for all $\theta\in\Theta_\theta$,
    $$Y^2(\theta;\tz)\leq -\chi_\theta|\theta-\theta_0|^2.$$
\end{Assumption}

% let $\Gamma={\rm diag}((\Gamma_1, \Gamma_2))$,
% where $S(x,\sigma)=b^2(x,\sigma)$, and
% \begin{align*}
%     [\Gamma_1]_{ij}&=\frac{1}{2}\int_{\mbbr^d} {\rm tr}(\partial_{\sigma_i}SS^{-1}\partial_{\sigma_j}SS^{-1})(x,\sigma_0)d\pi_0(x), \\
% [\Gamma_2]_{ij}&=\int_{\mbbr^d} (\partial_{\theta_i} a)^\top S^{-1}(\partial_{\theta_j} a)(x,\alpha_0)d\pi_0(x)
% + \int_{\mbbr^d} \frac{\partial_{\theta_i} f_{\theta_0}\partial_{\theta_j} f_{\theta_0}}{f_{\theta_0}} 1_{\{f_{\theta_0}\neq 0\}}(y)dy.
% \end{align*}

We leave some comments on our technical assumptions.
\begin{itemize}
\item From Assumption \ref{coef}, it is easy to check that the drift and diffusion coefficients are Lipschitz continuous, and thus, the Markov property and the existence of the strong solution $X^{\al,\lam}$ hold (cf. \cite{App09}).
\item Under Assumptions \ref{inidist}, \ref{coef}, \ref{nondege}, and \ref{jdist}, \cite[Theorem 1.4]{Mas08} also ensures the ergodic theorem:
\begin{equation}\label{ergthm}
    \frac{1}{T}\int_0^T f\big(X_t\big)dt\cip \int f(x)\pi_0(dx),
\end{equation}
for every $\pi_0$-integrable $f$. We can also see that $\sup_t E[|X_t^{\al,\lam}|^q]<\infty$ for all $\al$ and $\lam$ and that by means of the argument in \cite[Lemma 8]{Kes97}, under the balance condition: $nh_n\to\infty$ and $h_n\to0$ (looser than Assumption \ref{balance-cond}), we get
$$\frac{1}{n}\sumk f(X_{t_k})\cip \int f(x)\pi_0(dx),$$
for every differentiable function $f$ whose derivative and itself are of at most polynomial growth.
\item Regarding Assumption \ref{jdist}-(1), $\xi(\theta)\geq 0$ is straightforward by its definition, and for all $\kappa\in[0,\xi(\theta)]$, we have
\begin{align*}
    &\intd \frac{F_{\tz}(z)}{(F_\theta(z))^\kappa}1_{F_{\theta_0}\neq0}(z)dz\\
    &\leq \intd\frac{F_{\tz}(z)}{(F_\theta(z))^{\xi(\theta)}}1_{\{F_\theta(z)<1\}\cap\{F_{\theta_0}\neq0\}}(z)dz+\intd F_{\tz}(z)1_{\{F_\theta(z)\geq1\}\cap\{F_{\theta_0}\neq0\}}(z)dz<\infty.
\end{align*}
Assumption \ref{jdist}-(1) and the above estimate will be useful in the case where we evaluate remainder terms of the log likelihood function which typically have two different parameter in integrals with respect to the transition density function (for more details, see Lemma \ref{momratio} and Lemma \ref{derpar} in Section \ref{Appendix}). 
% Consider the situation where there exists one jump over the interval $(s,t)$, and then the increment $X_t-X_s$ can be decomposed as
% $$X_t-X_s=(X_t-X_\tau)+(X_\tau-X_{\tau-})+(X_{\tau-}-X_s),$$
% where $\tau$ denotes the jump time.
\item Examples of jump distributions and the construction of the truncating function $\varphi$ satisfying Assumption \ref{jdist} and Assumption \ref{jsupp} are presented in \cite{ShiYos06}, \cite{OgiYos11}, and \cite{OgiUeh23}.
\item The balance condition Assumption \ref{balance-cond} is stronger than the so-called ``rapidly increasing design": $nh_n\to\infty$ and $nh_n^2\to0$ often assumed in the ergodic framework in statistical inference of stochastic processes. This is for the moment convergence of the threshold-based quasi-likelihood estimator which will be introduced below, and for theoretical guarantees when using thresholds to discriminate jumps around 0. Especially, if the jump density $F_\theta(z)$ decreases at a polynomial order as $z$ tends to 0 (that is, $\gam$ takes a positive value in Assumption \ref{jdist}) or the dimension $d$ is larger than $2$, Assumption \ref{balance-cond}-(2) becomes looser.
\end{itemize}

\subsection{Quasi-likelihood analysis}
Since the explicit form of the transition density function of $X$ cannot be obtained in general, we instead consider the threshold-based quasi-likelihood function when we estimate unknown parameters $\lam$ and $\al=(\theta, \sig)$.
We briefly review the theory of the threshold quasi-likelihood function established in \cite{ShiYos06} and \cite{OgiYos11}.
Let $L_n=\{ x\in \mbbr^m||x|\leq h_n^\rho\}$ and $\rho\in(3/8+b,1/2)$ with a positive constant $b\in(0,1/8)$. 
We introduce the threshold quasi-likelihood function by \cite{ShiYos06} and \cite{OgiYos11}.
Following \cite{Man04} and \cite{ShiYos06}, we regard that at least one jump occurs within $(t_{k-1}, t_k]$ if the magnitude of the increment $\D_k X:= X_{t_k}-X_{t_{k-1}}$ exceeds the threshold $h_n^\rho$, and otherwise there is no jump on $(t_{k-1}, t_k]$, that is,
$$\D_k X\overset{P_{\al,\lam}}=\intk a(X_t,\theta)dt+\intk b(X_t,\sig)dw_t\overset{P_{\al,\lam}}\approx a(X_{t_{k-1}},\theta)h_n+b(X_{t_{k-1}},\sig)\D_k w.$$
We briefly write $f_k(\al)=f(X_{t_k},\al)$ for any function $f$.
By mimicking \cite{ShiYos06} and \cite{OgiYos11}, we define the threshold quasi-likelihood function $\mbbh_{n}(\al,\lam):=\mbbh_{1,n}(\al)+\mbbh_{2,n}(\lam,\theta)$ as:
\begin{align*}
    &\mbbh_{1,n}(\al)=-\frac{1}{2}\sumk\Bigg(\log \det S_{k-1}(\sig)+\frac{1}{h_n}S_{k-1}^{-1}(\sig)\Big[(\bar{X}_k(\theta))^{\otimes2}\Big]\Bigg)1_{L_n}(\D_k X),\\
    &\mbbh_{2,n}(\lam,\theta)=\sumk \Bigg(\log \lam +\log F_\theta(\D_k X)\Bigg)\varphi_n(\D_k X)1_{L_n^c}(\D_k X)-\lam nh_n\int F_\theta(z)\varphi_n(z)dz,
\end{align*}
where $\bar{X}_k(\theta)=\D_k X-h_na_{k-1}(\theta)$.
$\mbbh_{2,n}$ corresponds to the case where one jump occurs within $(t_{k-1}, t_k]$.
For a more detailed plausible construction of $\mbbh_n(\al,\lam)$, we refer to \cite[Section 2.3]{ShiYos06}.
We define the quasi-maximum likelihood estimator $(\aes,\les)$ as a maximizer of $\mbbh_{n}(\al,\lam)$.

% \tcb{From an easy calculation, the quasi maximum likelihood estimator of $\lam$ is explicitly given by
% \begin{align}\label{qmlelam}
%     \hat{\lam}_n=\frac{1}{nh_n\int_{\mbbr^d}F_{\hat{\theta}_n}(z)\varphi_n(z)dz} \sumk \varphi_n(\D_k X)1_{L_n^c}(\D_k X),
% \end{align}
% and especially when $\text{supp} (F_\theta)=\mbbr^d$, since we can set $\varphi_n(z)\equiv1$, $\hat{\lam}_n$ takes the following simpler form: 
% \begin{equation}\label{qmlelams}
%     \hat{\lam}_n=\frac{1}{nh_n}\sumk1_{L_n^c}(\D_k X).
% \end{equation}
% Notice that this estimator does not depend on the functional form of the drift and diffusion coefficients even in the general case.
% Although it depends on $F_{\hat{\theta}_n}$ via the integral with respect to the truncation function $\varphi$ in \eqref{qmlelam}, a simple calculation implies that $\hat{\lam}_n$ is asymptotically equivalent to the expression in \eqref{qmlelams} under Assumption \ref{truncate}.
% Hence, we can regard $\hat{\lam}_n$ as a model-free parameter.}
Let $\hat{u}_n=(\sqrt{n}(\hat{\sig}_n-\sig_0), \sqrt{T_n}(\hat{\theta}^\star_n-\theta_0^\star))$.
The asymptotic normality and moment convergence of $\hat{u}_n$ can be obtained under Assumptions \ref{inidist}-\ref{indentifiability}  (cf. \cite{ShiYos06} and \cite{OgiYos11}):
\begin{enumerate}
    \item (Asymptotic normality)
    $$\hat{u}_n\cil N(0,\Gam^{-1}).$$
    % where the asymptotic variance $\Sig=\diag\{\Sig_\sig^{-1},\Sig_{\theta^\star}^{-1}\}$ is expressed as
    % \begin{align*}
    %     &[\Sig_\sig]_{i,j}=\frac{1}{2}\int \text{tr}(\p_{\sig_i}SS^{-1}\p_{\sig_j} SS^{-1})(x,\sig_0)\pi(dx),\\
    %     &[\Sig_{\theta^\star}]_{i,j}=\int ((\p_{\theta_i} a)^\top S^{-1}(\p_{\theta_j} a))(x,\az)\pi(dx)+\int\frac{\p_{\theta_i^\star}f_{\tz,\lam}\p_{\theta_j^\star}f_{\tz,\lam}}{f_{\tz,\lam}}(z)dz
    % \end{align*}
    \item (Moment convergence) For any continuous function $f$ of at most polynomial growth, it follows that
    $$E[f(\hat{u}_n)]\to E[f(u)],$$
    where $u$ is a random variable whose distribution is $N(0,\Gam^{-1})$.
\end{enumerate}
Moreover, the quasi-maximum likelihood estimator is shown to be asymptotically efficient in the sense of the convolution theorem by \cite{OgiUeh23}.

% Model selection \tcr{write something}....

% \begin{align*}
%     &E\Bigg[\mbbh_{2,n}(\hat{\lam}_n,\tes)-\mbbh_{2,n}(\lam_0,\tz)\Bigg]\\
%     % &=\Bigg(\sumk P(\D_k Z\in L_n^c)\Bigg)E\Bigg[\log \frac{\hat{\lam}_n}{\lam_0}+h_n(\lam_0-\hat{\lam}_n)\Bigg]\\
%     % &+E\Bigg[\Ez\bigg[\sumk \Big(\log F_{\tes}(\D_k Z)-\log F_{\tz}(\D_k Z)\Big) 1_{L_n^c}(\D_k X)\bigg]\Bigg]
% \end{align*}
% Since the first term of the right-hand-side is model-free, we observe the second term from now on.

\subsection{Main result}
To evaluate the expected log-likelihood function, we first present the transition density estimates for jump diffusion processes.
We denote by $p_t(x,y;\al)$ the transition density function of $X^{\al,c}$ and by $q_t(x,y;\al,\lam)$ that of $X^{\al,\lam}$ where the process $X^{\al, c}$ is the solution of the following stochastic differential equation:
$$dX_t^{\al,c}=a(X_t^{\al,c},\theta)dt+b(X_t^{\al,c},\sig)dw_t.$$
Let $\tau_i$ be the $i$-th jump time of $X^{\al,\lam}$.
Since for each $i\geq2$, the conditional distribution of $(\tau_1,\dots,\tau_i)$ given the event $\{N_{h_n}=i\}$ is equivalent to that of the order statistics $U(1)\leq \dots \leq U(i)$ of i.i.d. $(0,h_n)$-uniformly distributed random variables \cite[Proposition 3.4]{Sat99}, we have
\begin{align}\label{dexpan}
    q_{h_n}(x,y;\al,\lam)=e^{-\lam h_n}&\Bigg(p_{h_n}(x,y;\al)+\sum_{i=1}^\infty q_{h_n}^{(i)}(x,y;\al)\Bigg),
\end{align}
where 
\begin{align}\label{tdecom}
    q_{h_n}^{(1)}(x,y;\al,\lam)=\lam\int_0^{h_n}\int\int p_{\tau_1}(x,x_1;\al) F_\theta(x_2) p_{h_n-\tau_1}(x_1+x_2, y; \al) dx_1dx_2d\tau_1
\end{align}
and for each $i\geq2$,
\begin{align*}
    &q_{h_n}^{(i)}(x,y;\al,\lam)\\
    &=\lam^i\int_0^{h_n}\int_{0}^{\tau_i}\dots\int_{0}^{\tau_2}\int\dots\int p_{\tau_1}(x,x_1;\al) F_\theta(x_2)\dots p_{h_n-\tau_i}(x_{2i-1}+x_{2i}, y; \al)\left(\prod_{l=1}^{2i} dx_l\right)\left(\prod_{l=1}^{i} d\tau_l\right).
\end{align*}

% We first introduce the estimates of $p_t(x,y;\al)$:
% \begin{Lem}(\cite[Proposition 1.2]{Gob02})\label{esg}
% Under Assumption \ref{}..., there exist constants $c>1$ and $K>1$ such that
% \begin{align*}
%     &p_t(x,y;\al)\leq \frac{K}{t^{d/2}}\exp\left(-\frac{|x-y|^2}{ct}\right)\exp(ct|x|^2),\\
%     &p_t(x,y;\al)\geq \frac{1}{Kt^{d/2}}\exp\left(-c\frac{|x-y|^2}{t}\right)\exp(-ct|x|^2),
% \end{align*}
% for all $0<t\leq 1, \ (x,y)\in\mbbr^d\times\mbbr^d$, and $\al\in\Theta$.
% \end{Lem}

% \begin{Rem}
%     Under Assumption \ref{coef}, the following Aronson-type estimates hold for $p_t(x,y;\al)$: for all $\al\in\Theta_\al:=\Theta_\theta\times \Theta_\sig$ and $T>0$, there exists a positive constant $C_T>1$ such that for all $t\in(0,T)$, 
%     $$\frac{1}{(2\pi)^{d/2}C_T}\exp\bigg[-\frac{C_T|x-y|^2}{2}\bigg]\leq p_t(x,y;\al)\leq \frac{C_T}{(2\pi)^{d/2}}\exp\bigg[-\frac{|x-y|^2}{2C_T}\bigg].$$
%     Together with \eqref{dexpan}, the transition density $q_t(x,y;\al,\lam)$ always takes a positive value for every $t>0$. 
% \end{Rem}

\begin{Rem}
    \cite{OgiUeh23} also uses the decomposition \eqref{tdecom} to show local asymptotic normality.
    In the paper, it suffices to observe the total variation norm between the original probability measure based on the transition density and its approximated one constructed by $p_t(x,y;\al)$ and $q^{(1)}(x,y;\al,\lam)$.
    However, since we will evaluate the expectation of the log likelihood function below, we need more precise estimates on 
    $$\bar{q}^{(l)}_{h_n}(x,y;\al,\lam)=e^{-\lam h_n}\sum_{i=l}^\infty q_{h_n}^{(i)}(x,y;\al).$$
    for $l\geq 2$.
\end{Rem}

The next theorem gives the exponentially decaying upper bound of $q_{h_n}^{(i)}(x,y;\al)$.
\begin{Thm}\label{ktesti}
Under Assumptions \ref{inidist}-\ref{balance-cond}, there exist positive constants $C$, $\zeta<1$ and $N\in\mbbn$ such that for all $x,y\in\mbbr^d$, $k\in\mbbn$ and $n\geq N$, 
\begin{align}
    \sup_{\al\in\Theta}\sup_{\lam\in(\underline{\lam},\bar{\lam})}q_{h_n}^{(k)}(x,y;\al,\lam)\leq (Ch_n)^k\exp\left(-u\zeta||y-x||_1\right).
\end{align}
\end{Thm}

By taking the sum and using Aronson-type estimates on $p_t(x,y;\al)$, we can easily obtain the following transition density estimates:
\begin{Cor}\label{testi}
    Under Assumptions \ref{inidist}-\ref{balance-cond}, there exist positive constants $C$, $C'$ and $N\in\mbbn$ such that for all $x,y\in\mbbr^d$, $k\in\mbbn$ and $n\geq N$, 
    \begin{align}
    &\sup_{\al\in\Theta}\sup_{\lam\in(\underline{\lam},\bar{\lam})}\bar{q}^{(2)}_{h_n}(x,y;\al,\lam)\leq Ch_n^2\exp\left(-u\zeta||y-x||_1\right),\\
            &\sup_{\al\in\Theta}\sup_{\lam\in(\underline{\lam},\bar{\lam})}q_{h_n}(x,y;\al,\lam)\leq C\left[\frac{1}{h_n^{d/2}}\exp\left(-\frac{|y-x|^2}{C'h_n}\right)+h_n\exp\left(-u\zeta||y-x||_1\right)\right].
    \end{align}
\end{Cor}

\begin{Rem}
For jump diffusion models, \cite[Lemma 1]{AmoGlo21} provides the transition density estimates for $q_{t}(x,y;\al,\lam)$ based on \cite[Theorem 1.1]{CheHuXieZha17}, and clarifies the asymptotic property of nonparametric estimators for the invariant density function.
\cite{KohNuaTra22} also gives the transition density estimates for $q_{t}(x,y;\al,\lam)$ with a lower bound in the case where the jump distribution $F$ is a standard normal distribution or a centered Laplace distribution.
These estimates are more general than our result in the sense that they hold for any positive $t$, but if we consider high-frequency observations and look at the separated events conditioned on jump times, Theorem \ref{ktesti} and Corollary \ref{testi} will be useful.
\end{Rem}

We define the threshold-based approximating function by
\begin{align*}
     \tilde{q}_k(\alpha,\lam)&=\tilde{q}_{h_n}(z_k,z_{k-1};\al,\lam)\\
      &=e^{-\lam h_n}\Big[p_{h_n}(z_{k-1},z_k;\alpha)1_{L_n}(\Delta z_k)+q^{(1)}_{h_n}(z_{k-1},z_k;\alpha,\lam)1_{L_n^c}(\Delta z_k)\Big],
\end{align*}
where $\Delta z_k=z_k-z_{k-1}$.
Let $\pi_k(\cdot)$ be the probability distribution of $X_{t_k}$ for any $k\in\mbbn$.
Following the usual context of minimizing the Kullback-Leibler divergence, we regard the best model among the candidates as $\kl(\mbfx_n)=\Ez\Big[l_n(\mbfz_n,\tilde{\al}_n,\tilde{\lam}_n)\Big]$ for some estimators $\tilde{\al}_n=\tilde{\al}_n(\mbfx_n)$ and $\tilde{\lam}_n=\tilde{\lam}_n(\mbfx_n)$.
The following proposition clarifies the main term of the expected Kullback-Leibler divergence with either the true values or the threshold-based QMLE as an assignment. 

\begin{Prop}\label{eKL}
    Under Assumptions \ref{inidist}-\ref{indentifiability}, we have
\begin{align}
    &\label{app1}\Ez\Big[l_n(\mbfz_n,\az,\lz)\Big]=\sumk\intd\intd\log \tilde{q}_k(\al_0,\lam_0) \tilde{q}_k(\al_0,\lam_0)dz_k \pi_{k-1}(dz_{k-1})+o(1),\\
    &\label{app2}E\Bigg[\Ez\Big[l_n(\mbfz_n,\aes,\les)\Big]-\sumk\intd\intd\log \tilde{q}_k(\aes,\les) \tilde{q}_k(\al_0,\lam_0)dz_k \pi_{k-1}(dz_{k-1})\Bigg]=o(1).
\end{align}
\end{Prop}
We write 
\begin{align*}
    &\ic(\mbfx_n)=\mbbh_n(\aes,\les)-\mbbh_n(\az,\lz)+\Ez\Big[l_n(\mbfz_n,\az,\lz)\Big]-\dim(\Theta)-1,\\
    &\el(\mbfx_n)=\Ez\Big[l_n(\mbfz_n,\aes,\les)\Big].
\end{align*}

The next theorem shows that $\ic(\mbfx_n)$ is an asymptotically unbiased estimator of $\el(\mbfx_n)$.
\begin{Thm}\label{asmun}
    Under Assumptions \ref{inidist}-\ref{indentifiability}, we have
\begin{equation}
        E\Bigg[\ic(\mbfx_n)-\el(\mbfx_n)\Bigg]=o(1).
\end{equation}
\end{Thm}

Although we cannot directly calculate $\ic(\mbfx_n)$ only by observations, the quantities $\mbbh_n(\az,\lz)$, and $\Ez\Big[l_n(\mbfz_n,\az,\lz)\Big]$ do not depend on the candidate models under the nested condition given in the introduction.
% On the contrary, since the quasi-maximum likelihood estimator of $\lam$ is explicitly given by
% \begin{align*}\label{qmlelam}
%     \hat{\lam}_n=\frac{1}{nh_n\int_{\mbbr^d}F_{\hat{\theta}_n}(z)\varphi_n(z)dz} \sumk \varphi_n(\D_k X)1_{L_n^c}(\D_k X),
% \end{align*}
% $\tilde{b}_{1,\lam}$ and $\tilde{b}_{2,\lam}$ seem to have some candidate model-specific quantity.
% However, 
Hence, following the conventional AIC, we define our predictive information criterion (PIC) by
\begin{equation}
\text{PIC}=-2\mbbh_n(\aes,\les)+2\dim(\Theta),
\end{equation}
and we select the model which minimizes PIC.

In the following, we consider the relative comparison of two candidate models:
\begin{align*}
&\mcm: dX_t=a(X_t,\theta)dt+b(X_t,\sig)dw_t+\int zN_{\lam,\theta}(dt,dz),\\
&\mcm_\star: dX_t=a(X_t,\theta^{(\star)})dt+b(X_t,\sig^{(\star)})dw_t+\int zN_{\lam,\theta^{(\star)}}(dt,dz).
\end{align*} 
We assume an additional nested condition between $\mcm$ and $\mcm_\star$: $$p_\theta:=\dim(\Theta_\theta)>\dim(\Theta_{\theta^{(\star)}})=:p_{\theta^{(\star)}}, \ p_\sig:=\dim(\Theta_\sig)>\dim(\Theta_{\sig^{(\star)}})=:p_{\sig^{(\star)}},$$
and there exist vectors $c_\theta\in\mbbr^{p_\theta}$ and $c_\sig\in\mbbr^{p_\sig}$, and matrices $F_\theta\in\mbbr^{p_\theta\times p_{\theta^{(\star)}}}$ and $F_\sig\in\mbbr^{p_\sig\times p_{\sig^{(\star)}}}$ with $F_\theta^\top F_\theta=I_{p_{\theta^{(\star)}}}$ and $F_\sig^\top F_\sig=I_{p_{\sig^{(\star)}}}$ such that for any $\al^{(\star)}\in\Theta_{\sig^{(\star)}}\times \Theta_{\theta^{(\star)}}$,
$$\mbbh^{(\mcm_\star)}_n(\al^{(\star)},\lam)=\mbbh^{(\mcm)}_n(F\al^{(\star)}+c,\lam),$$
where $F=\diag\{F_\theta,F_\sig\}\in\mbbr^{p_\mcm\times p_{{\mcm}_\star}}$, $c=(c_\theta,c_\sig)$,
$\mbbh^{(\mcm)}_n(\cdot,\cdot)$ stands for the quasi-likelihood function of the statistical model $\mcm$, $p_\mcm=p_\theta+p_\sig$ and $p_{\mcm_\star}=p_{\theta^{(\star)}}+p_{\sig^{(\star)}}$.

\begin{Thm}\label{relpro}
    Suppose that Assumptions \ref{inidist}-\ref{indentifiability} hold for both statistical models $\mcm$ and $\mcm_\star$.
    Then we have
\begin{align*}
&\lim_{n\to\infty}P(\text{PIC}_{\mcm_\star}-\text{PIC}_{\mcm}>0)=P(\chi^2(p_{\mcm}-p_{\mcm_\star})>2(p_{\mcm}-p_{\mcm_\star}))>0,
\end{align*}
where $\chi^2(p_{\mcm}-p_{\mcm_\star})$ denotes the chi-square random variable with $p_{\mcm}-p_{\mcm_\star}$ degrees of freedom.
\end{Thm}
 
\begin{Rem}
    This theorem indicates that our PIC asymptotically selects a redundant model with positive probability as the conventional AIC.
\end{Rem}

\begin{Rem}
    When the asymptotic variance of the estimator takes a sandwich form, the chi-square random variable in the last probability would be a non-central chi-square random variable (cf. Fasen and Kimming (2017) and Eguchi and Masuda (2023)).
\end{Rem}

\begin{Rem}\label{missprob}
We consider the situation where the statistical model $\mcm_{miss}$:
\begin{equation*}
    \mcm_{miss}: dX_t=\bar{a}(X_t,\theta)+\bar{b}(X_t,\sig)dw_t+\int z \bar{N}_{\lam,\theta}(dt,dz),
\end{equation*}
is (semi-)misspecified, that is, for any $\theta$,  
$\int_{\bar{a}(x,\sig)\neq A(x)}\pi_0(dx)>0,$
or if for any $\sig$, 
$\int_{\bar{b}(x,\sig)\neq B(x)}\pi_0(dx)>0,$
or if for any $\theta$, 
$\int_{F_\theta(z)\neq F(z)}F(z)dz>0$.
For such a model, the asymptotic behavior of the threshold-based QMLE has not been not investigated yet.
However, we can formally use our PIC for model comparison as a matter of course, and its validity is ensured by the fact that the probability of selecting the above misspecified model tends to zero:
\begin{align}
    \label{probmiss}\lim_{n\to\infty} P(PIC_{\mcm}-PIC_{\mcm_{miss}}>0)=0.
\end{align} 
The outline of the proof of \eqref{probmiss} is as follows.
Let $S(x)=(B(x))^{\otimes 2}$ and $\bar{S}(x,\sig)=(\bar{b}(x,\sig))^{\otimes2}$.
Introduce
\begin{align*}
    &\bar{Y}_1(\sig)=-\frac{1}{2}\int \Tr\big(\bar{S}^{-1}(x,\sig)S(x)\big)\pi_0(dx)-\frac{1}{2}\int\log\det \bar{S}(x,\sig)\pi_0(dx),\\
    &\bar{Y}_2(\theta)=-\frac{1}{2}\int \bar{S}^{-1}(x,\sig_{opt})\Big[\big(\bar{a}(x,\theta)-A(x)\big)^{\otimes2}\Big]\pi_0(dx)+\int \log \bar{F}_\theta(z)F(z)dz.
\end{align*}
Here, $\sig_{opt}$ is a maximizer of $\bar{Y}_1$, and we also use $\theta_{opt}$ as a maximizer of $\bar{Y}_2$, and both $\bar{Y}_1$ and $\bar{Y}_2$ are the probability limit of the quasi-likelihood function.
Applying the standard M-estimation theory and the estimates in \cite{ShiYos06} and \cite{OgiYos11}, we can show that under suitable regularity and identifiability conditions, the threshold-based QMLE has the consistency: $\aes\cip\al_{opt}=(\theta_{opt},\sig_{opt})$.
Notice that under the (semi-)misspecification condition, we have 
$$\bar{Y}_1(\sig_{opt})+\frac{1}{2}d+\frac{1}{2}\int\log\det S(x,\sig_0)\pi_0(dx)<Y_1(\sig_0),$$
or $\bar{Y}_2(\theta_{opt})+\int \log F(z)F(z)dz<Y_2(\tz)$.
Hence, by dividing the difference $PIC_{\mcm}-PIC_{\mcm_{miss}}$ by $n$ or $nh_n$ before taking the limit, we obtain \eqref{probmiss}.
\end{Rem}

\section{Numerical Experiment}\label{Numerical}
In this section, we suppose that the data generating process is given by:
\begin{equation}\label{simmod}
    dX_t=-\frac{X_t^2}{1+X_t^2}dt+\frac{3+2X_t^2}{1+X_t^2}dw_t+dZ_t,
\end{equation}
where $Z_t$ is the compound Poisson process whose intensity is $1$, and jump distribution is the Laplace distribution whose probability density function is
$$f(z)=\frac{1}{4}\exp\Bigg[-\frac{|z|}{2}\Bigg].$$

We independently simulate $1000$ paths from \eqref{simmod} for $(n,T_n)=(1000,10), (5000, 50),(10000,100)$.
For the observations from \eqref{simmod}, we consider the  following candidate coefficients:
\begin{align*}
    &\text{Drift 1}: a_1(x,\theta_1)=\frac{\theta_{11}x^2+\theta_{12}x+\theta_{13}}{1+x^2}; \ \text{Drift 2}: a_2(x,\theta_2)=\frac{\theta_{21}x^2+\theta_{22}}{1+x^2}; \\
    &\textbf{Drift 3}: a_3(x,\theta_3)=\frac{\theta_{31}x^2}{1+x^2},\\
    &\text{Diffusion 1}: b_1(x,\beta_1)=\frac{\beta_{11}x^2+\beta_{12}x+\beta_{13}}{1+x^2}; \ \textbf{Diffusion 2}: b_2(x,\beta_2)=\frac{\beta_{21}x^2+\beta_{22}}{1+x^2},\\
    &\text{Jump 1}: F_{1,\theta'_1}(z)=\frac{1}{\sqrt{2\pi}\theta'_{12}}\exp\Bigg[-\frac{(z-\theta'_{11})^2}{2(\theta'_{12})^2}\Bigg],\\
    &\textbf{Jump 2}: F_{2,\theta'_2}(z)=\frac{1}{2 \theta'_{22}}\exp\Bigg[-\frac{|z-\theta'_{21}|}{2 \theta'_{22}}\Bigg]
    \end{align*}
We set the threshold as $h^{2/5}$ and see the performance of the proposed PIC for each case.
In this case, since the parameters in the drift and jump components are completely separated, the selection results are shown separately in Table \ref{tab1} and Table \ref{tab2}.
As can be seen in Table \ref{tab1}, the proposed PIC selects redundant models to some extent, and this result confirms Theorem \ref{relpro} from the numerical point of view.
In addition, since Jump 1 is misspecified, as is stated in Remark \ref{missprob}, the number of selecting Jump1 decreases as $T_n$ increases and $h_n$ decreases.

\begin{table}
    \centering
    \begin{tabular}{ccccc}\hline
         $T_n$& $h_n$ &  & Diffusion 1 & Diffusion 2\\\hline
        30 & 0.05 & Drift 1 &3 & 37\\
        &&Drift 2& 4 & 86 \\
        &&Drift 3& 21 &\textbf{849} \\
        50 & 0.025 & Drift 1 & 1&30 \\
        &&Drift 2& 2 &76 \\
        &&Drift 3& 23 &\textbf{868} \\
        100 & 0.01 & Drift 1 & 2 & 30\\
        &&Drift 2& 3 & 84\\
        &&Drift 3& 34 & \textbf{847}\\\hline
    \end{tabular}
    \caption{Selection result for the drift and diffusion coefficients}
    \label{tab1}
\end{table}

\begin{table}
    \centering
    \begin{tabular}{cccc}\hline
         $T_n$& $h_n$  & Jump 1 & Jump 2\\\hline
        30&0.05& 173 & \textbf{827}\\
        50 & 0.025 & 141 & \textbf{859}\\
        100 & 0.01 & 49  & \textbf{951}\\\hline
    \end{tabular}
    \caption{Selection result for the jump density}
    \label{tab2}
\end{table}

% \clearpage
\section{Appendix}\label{Appendix}
In this section, we present the proof of our main results.
% First, we introduce some existing results which we will use repeatedly in our proof.
% \begin{itemize}
% \item Aronson-type estimates: there exists a positive constant $C>1$ such that for all $\al\in\Theta_\al:=\Theta_\theta\times \Theta_\sig$, and $n\in\mbbn$, 
%     $$\frac{1}{(2\pi h_n)^{d/2}C}\exp\bigg[-\frac{C|x-y|^2}{2h_n}\bigg]\leq p_{h_n}(x,y;\al)\leq \frac{C}{(2\pi h_n)^{d/2}}\exp\bigg[-\frac{|x-y|^2}{2Ch_n}\bigg].$$
% \item \cite[Lemma 2]{OgiYos11} and \cite[p.2356-2357]{OgiUeh23}: for any $x_{k-1}$, there exists a positive constant $C$ such that
% \begin{equation}\label{opjump}
%     P(N_\theta((t_{k-1},t_k]\times\mbbr^d)=1 \ \text{and} \ |\D_k X|\leq h_n^\rho|X_{t_{k-1}}=x_{k-1})\leq Ch_n^{1+(d+\gam)\rho\wedge1}(1+|x_{k-1}|^C).
%     % \int_{\mbbr^d}\int_{\mbbr^d} q_k^{(1)}(\al_0,\lz)1_{L_n}(\D x_k)dx_k
% \end{equation}
% \item Application of \cite[Lemma 2]{Uch10}: 
% \item 
% \end{itemize}
In the proofs, $C$ also denotes a generic positive constant that possibly varies from line to line.

% Before we proceed to state our result, we introduce some convenient results.
\subsection{Proof of Theorem \ref{ktesti}}

% \begin{Rem}\label{nesg}
%     Since $ct\leq 1/(4ct)$ for a small enough $t$, it follows that for any $(x,y)\in\mbbr^d\times\mbbr^d$,
%     \begin{align*}
%     p_t(x,y;\al)&\leq \frac{K}{t^{d/2}}\exp\left(-\frac{|x-y|^2}{ct}\right)\exp(ct|x|^2),\\
%     &\leq \frac{K}{t^{d/2}}\exp\left(-\frac{|x-y|^2}{ct}\right)\exp(2ct|x-y|^2)\exp(2ct|y|^2)\\
%     &\leq \frac{K}{t^{d/2}}\exp\left(-\frac{|x-y|^2}{2ct}\right)\exp(2ct|y|^2),
%     \end{align*}
%     for such a $t$. 
%     Similarly we can also derive the lower bound:
%     $$p_t(x,y;\al)\geq \frac{1}{Kt^{d/2}}\exp\left(-2c\frac{|x-y|^2}{t}\right)\exp(-2ct|y|^2)$$
%     for any $(x,y)\in\mbbr^d\times\mbbr^d$ and a small enough $t$. 
%     We will also use these estimates below.
% \end{Rem}

Before proceeding to the proof of Theorem \ref{ktesti}, we prepare some lemmas below.

\begin{Lem}
    Let $\{a_n\}$ be a positive and decreasing sequence with $a_n\to0$.
    Then, for all $u>0$, there exist positive constants $C$ and $N\in \mbbn$ such that for all $z\in\mbbr$ and $n\geq N$,
    \begin{equation}
        \int_\mbbr a_n^{-1/2}\exp(-u|x|)\exp\left(-\frac{|x-z|^2}{a_n}\right)dx\leq C \exp(-u|z|).
    \end{equation}
\end{Lem}
\begin{proof}
By the change of variables, we have
\begin{align*}
    &\int_\mbbr a_n^{-1/2}\exp(-u|x|)\exp\left(-\frac{|x-z|^2}{a_n}\right)dx\\
    &=a_n^{-1/2}\exp(u^2a_n/2)\Bigg[\exp(-uz)\int_{0}^\infty\exp\left(-\frac{|x-(z-ua_n)|^2}{a_n}\right) dx\\ &\qquad\qquad\qquad\qquad\qquad+\exp(uz)\int_{-\infty}^0\exp\left(-\frac{|x-(z+ua_n)|^2}{a_n}\right) dx\Bigg]\\
    &=\exp(u^2a_n/2)\Bigg[\exp(-uz)\int_{-a_n^{-1/2}(z-ua_n)}^\infty\exp\left(-|v|^2\right) dv\\ &\qquad\qquad\qquad\qquad\qquad+\exp(uz)\int_{-\infty}^{-a_n^{-1/2}(z+ua_n)}\exp\left(-|v|^2\right) dv\Bigg].
\end{align*}
From now on, we show that there exists a positive constant $L$ satisfying
\begin{equation}\label{lapnor}
    \exp(-uz)\int_{-a_n^{-1/2}(z-ua_n)}^\infty\exp\left(-|v|^2\right) dv\leq L \exp[-u|z|],
\end{equation}
for all large enough $n\geq N$.
For notational brevity, we write
\begin{align*}
    &D_{1,n}=\{z|-a_n^{-1/2}(z-ua_n)\leq0\},\\
    &D_{2,n}=\{z|0< -a_n^{-1/2}(z-ua_n)< 1\},\\
    &D_{3,n}=\{z|-a_n^{-1/2}(z-ua_n)\geq 1\}.
\end{align*}
It is straightforward from $z1_{D_{1,n}}(z)\geq0$ that
\begin{align*}
    &\exp(-uz)\int_{-a_n^{-1/2}(z-ua_n)}^\infty\exp\left(-|v|^2\right) dv1_{D_{1,n}}(z)\lesssim \exp(-u|z|)1_{D_{1,n}}(z).
\end{align*}
Since $a_n\to0$ and
\begin{align*}
    0< -a_n^{-1/2}(z-ua_n)< 1\Leftrightarrow -a_n^{1/2}+ua_n<z<ua_n,
\end{align*}
there exists $N\in\mbbn$ such that for all $n\geq N$,
\begin{align*}
    &\max_{z\in D_{2,n}}\exp(-uz)=\exp[-u(-a_n^{1/2}+ua_n)]\leq 2,\\
    &\min_{z\in D_{2,n}}\exp(-u|z|)=\exp(-u^2a_n)\wedge \exp[u(-a_n^{1/2}+ua_n)]\geq\frac{1}{2}.
\end{align*}
This leads to
\begin{align*}
    &\exp(-uz)\int_{-a_n^{-1/2}(z-ua_n)}^\infty\exp\left(-|v|^2\right) dv1_{D_{2,n}}(z)\\
    &\lesssim \frac{\exp(-uz)}{ \exp(-u|z|)} \exp(-u|z|)1_{D_{2,n}}(z)\\
    &\leq 4\exp(-u|z|)1_{D_{2,n}}(z).
\end{align*}
Recall that for all $y\geq1$, we have
\begin{align*}
    \int_y^\infty \exp\left(-|v|^2\right) dv\leq (2y)^{-1}\exp\left(-|y|^2\right)\leq \exp\left(-|y|^2\right).
\end{align*}
By using this estimate and elementary inequalities: $|x+y|^2\geq |x|^2/2-|y|^2$ and $|xy|^2\geq 2|x||y|-1$ for all $x,y\in\mbbr$, it follows that
\begin{align*}
    &\exp(-uz)\int_{-a_n^{-1/2}(z-ua_n)}^\infty\exp\left(-|v|^2\right) dv1_{D_{3,n}}(z)\\
    &\leq \exp(-uz)\exp\left(-|a_n^{-1/2}(z-ua_n)|^2\right)1_{D_{3,n}}(z)\\
    &\leq \exp(-uz)\exp\left(-\frac{a_n^{-1}|z|^2}{2}\right)\exp\left(u^2a_n\right)1_{D_{3,n}}(z)\\
    &\leq \exp(-uz)\exp\left(-a_n^{-1/2}|z|\right)\exp\left(u^2a_n+\frac{1}{2}\right)1_{D_{3,n}}(z).
\end{align*}
Again using $a_n\to0$, we can find a positive constant $K$ such that
$$\exp(-uz)\int_{-a_n^{-1/2}(z-ua_n)}^\infty\exp\left(-|v|^2\right) dv1_{D_{3,n}}(z)\leq K\exp(-u|z|)1_{D_{3,n}}(z),$$
for all large enough $n$.
Thus, we get \eqref{lapnor}.
In a similar way, we can also derive
\begin{align*}
    \exp(uz)\int_{-\infty}^{-a_n^{-1/2}(z+ua_n)}\exp\left(-|v|^2\right) dv\leq L\exp(-u|z|),
\end{align*}
for a suitable positive constant $L$ and all large enough $n$.
Hence, the desired result follows.

\end{proof}

Let $\{b_k\}$ be a positive and decreasing sequence with $b_1<1$ and $\lim_{k\to\infty} b_k=\zeta>0$.

\begin{Lem}\label{lkconv}
    Let $\{a_n\}$ be a positive and decreasing sequence with $a_n\to0$.
    Then, for all $u>0$, there exists an $N\in \mbbn$ such that for all $z\in\mbbr$ and $n\geq N$,
    \begin{align}\label{kconv}
        \int_\mbbr\dots\int_\mbbr a_n^{-1/2}\exp\left(-u\sum_{i=1}^k|x_i|\right)\exp\left(-\frac{|\sum_{i=1}^kx_i-z|^2}{a_n}\right)\prod_{i=1}^kdx_i\leq (C\vee 2u)^k \exp(-ub_k|z|),
    \end{align}
    where the positive constant $C$ is the same as in the previous lemma.
\end{Lem}

\begin{proof}
    In the case where $k=1$, \eqref{kconv} is straightforward from the previous lemma.
    Suppose that for $k=l$, \eqref{kconv} holds true with $n\geq N_l\in\mbbn$.
    Then, for $k=l+1$ and $n\geq N_l$, the triangular inequality leads to
    \begin{align*}
        &\int_\mbbr\dots\int_\mbbr a_n^{-1/2}\exp\left(-u\sum_{i=1}^{l+1}|x_i|\right)\exp\left(-\frac{|\sum_{i=1}^{l+1}x_i-z|^2}{a_n}\right)\prod_{i=1}^{l+1}dx_i\\
        &\leq (C\vee 2u)^l\int_\mbbr \exp\left(-u|x_{l+1}|\right)\exp\left(-ub_l|x_{l+1}-z|\right)dx_{l+1}\\
        &\leq (C\vee 2u)^l \exp\left(-ub_l|z|\right)\int_\mbbr \exp\left(-u\left[1-b_l\right]|x_{l+1}|\right)dx_{l+1}\\
        &=2u(C\vee 2u)^l\left(1-b_l\right)\exp\left(-ub_l|z|\right)\\
        &\leq (C\vee 2u)^{l+1}\exp\left(-ub_{l+1}|z|\right).
    \end{align*}
    Thus, the desired result follows.
\end{proof}
We note that since $b_k\to \zeta$ as $k\to\infty$, the right-hand-side of \eqref{kconv} is bounded by $(C\vee 2u)^k \exp(-u\zeta|z|)$ for any $k$.
% Let $M_n=\exp[(ch_n|u|)^2/8]$ and by definition of $h_n$, it is straightforward that $\sup_{n\in\mbbn}M_n<\infty$.
Now we prove Theorem \ref{ktesti}.
For simplicity, we set $d=1$.
From Aronson-type estimates,
\begin{align*}
        &q_{h_n}^{(1)}(x,y;\al)\\
        &=\lam\int_0^{h_n}\int_\mbbr\int_\mbbr p_{\tau_1}(x,x_1;\al) F_\theta(x_2) p_{h_n-\tau_1}(x_1+x_2, y; \al) dx_1dx_2d\tau_1\\
        &\lesssim \int_0^{h_n}\int_\mbbr\int_\mbbr \tau_1^{-1/2}\exp\left(-\frac{|x_1-x|^2}{C\tau_1}\right)\exp(-u|x_2|) (h_n-\tau_1)^{-1/2}\exp\left(-\frac{|(y-x_1)-x_2)|^2}{C(h_n-\tau_1)}\right)dx_1dx_2d\tau_1\\
        % &=C\int_0^{h_n}\int\exp(-u|x_2|)\int \tau_1^{-1/2}\exp\left(-\frac{|x_1-x|^2}{c\tau_1}\right) (h_n-\tau_1)^{-1/2}\exp\left(-\frac{|y-(x_1+x_2)|^2}{c(h_n-\tau_1)}\right)dx_1dx_2d\tau_1\\
        % &\leq C\int_0^{h_n}\int_\mbbr\exp(-u|x_2|)\int_\mbbr \tau_1^{-1/2}\exp\left(-\frac{|v|^2}{C'\tau_1}\right) (h_n-\tau_1)^{-1/2}\exp\left(-\frac{|y-x_2-x-v|^2}{C'(h_n-\tau_1)}\right)dvdx_2d\tau_1\\
        &=\int_0^{h_n}\int_\mbbr h_n^{-1/2}\exp(-u|x_2|)\exp\left(-\frac{|x_2-(y-x)|^2}{Ch_n}\right)dx_2d\tau_1.
\end{align*}
In the last equality, we used the fact that the sum of two independent normal random variables also obeys normal distribution whose mean and variance are the sum of the original variables.
By using a similar estimate twice, we have
\begin{align*} 
    &\int_\mbbr \int_\mbbr \prod_{l=1}^3 p_{\tau_{i+l}-\tau_{i+l-1}}(x_{2i+2l-3}+x_{2i+2l-2},x_{2i+2l-1};\al)dx_{2i+1}dx_{2i+3}\\
    &\lesssim \int_\mbbr (\tau_{i+2}-\tau_i)^{-1/2}\exp\left(-\frac{|x_{2i+3}-(x_{2i-1}+x_{2i}+x_{2i+2})]|^2}{C(\tau_{i+2}-\tau_i)}\right)p_{\tau_{i+2}}(x_{2i+3}+x_{2i+4},x_{2i+5};\al)dx_{2i+3}\\
    &\lesssim  (\tau_{i+3}-\tau_i)^{-1/2}\exp\left(-\frac{|x_{2i+5}-(x_{2i-1}+x_{2i}+x_{2i+2}+x_{2i+4})]|^2}{C(\tau_{i+3}-\tau_i)}\right).
\end{align*}
Hence, by induction, we obtain
\begin{align*}
    &\int_\mbbr\dots\int_\mbbr p_{\tau_1}(x,x_1;\al) \left(\prod_{i=1}^{k-1} p_{\tau_{i+1}-\tau_i}(x_{2i-1}+x_{2i},x_{2i+1};\al)\right)p_{h_n-\tau_k}(x_{2k-1}+x_{2k},y;\al)\prod_{i=1}^{k}dx_{2i-1}\\
    &\lesssim h_n^{-1/2}\exp\left(\frac{|\sum_{i=1}^kx_{2l}- (y-x)|^2}{Ch_n}\right).
\end{align*}
Combined with Lemma \ref{lkconv}, we have
\begin{align*}
    &q^{(k)}_{h_n}(x,y;\al,\lam)\\
    &\lesssim C h_n^k\int_\mbbr \dots\int_\mbbr h_n^{-1/2}\exp\left(-u\sum_{i=1}^k |x_{2i}|\right)\exp\left(\frac{|\sum_{i=1}^kx_{2l}- (y-x)|^2}{Ch_n}\right)\prod_{i=1}^{k}dx_{2i}\\
    &\lesssim h_n^k \exp\left(-u\zeta|y-x|\right),
\end{align*}
for all large enough $n$ and this implies the desired result for $d=1$.
Since in the multi-dimensional case the proof is quite similar to the one-dimensional case, we omit it.

\subsection{Proof of Proposition \ref{eKL}}

Recall that we write $L_n=\{ z\in \mbbr^d||z|\leq h_n^\rho\}$.
We define the threshold-based approximating function by
\begin{align*}
     \tilde{q}_k(\alpha,\lam)&=\tilde{q}_{h_n}(z_k,z_{k-1};\al,\lam)\\
      &=e^{-\lam h_n}\Big[p_{h_n}(z_{k-1},z_k;\alpha)1_{L_n}(\Delta z_k)+q^{(1)}_{h_n}(z_{k-1},z_k;\alpha,\lam)1_{L_n^c}(\Delta z_k)\Big],
\end{align*}
where $\Delta z_k=z_k-z_{k-1}$.
We also write 
$$p_k(\al)=p_{h_n}(z_{k-1},z_k;\alpha),\quad q_k^{(1)}(\al,\lam)=q_{h_n}^{(1)}(z_{k-1},z_k;\al,\lam),\quad q_k(\al,\lam)=q_{h_n}(z_{k-1},z_k;\al,\lam),$$
and by definition, $\tilde{q}_k(\al,\lam)\leq q_k(\al,\lam)$.
Let $\pi_k(\cdot)$ be the distribution of $Z_k$. 
Before we show \eqref{app1}, we prepare the following lemma:

\begin{Lem}\label{lemapp1}
    Under Assumptions \ref{inidist}-\ref{indentifiability}, we have
    \begin{align*}
    \ezn[l_n(\mbfz_n,\al_0,\lam_0)]=\sumk\intd\intd\log q_k(\al_0,\lam_0) \tilde{q}_k(\al_0,\lam_0)dz_k \pi_{k-1}(dz_{k-1})+o(1).
    \end{align*}
\end{Lem}

\begin{proof}
Notice that
\begin{align}
\nn &P(N_{\tz}((t_{k-1},t_k]\times\mbbr^d)=1 \ \text{and} \ |\D_k Z|\leq h_n^\rho|Z_{t_{k-1}}=z_{k-1})\\
\nn&= \int_{\mbbr^d} q_k^{(1)}(\al_0,\lz)1_{L_n}(\D z_k)dz_k\\
\label{pj1c} &\lesssim h_n^{1+(d+\gam)\rho\wedge1}(1+|z_{k-1}|^C).
\end{align}
Then, by taking a similar route to the proof of \cite[(4.3)]{OgiUeh23}, it follows that for any $z_{k-1}\in\mbbr^d$,
\begin{align}
\label{tvd}&\int_{\mbbr^d}\int_{\mbbr^d} (q_{k}(\al_0,\lam_0)-\tilde{q}_k(\al_0,\lam_0))dx_k\pi_{k-1}(dx_{k-1})\\
    \nn&\leq\int_{\mbbr^d} \bigg[p_k(\az)1_{L_n^c}(\Delta z_k)+q_k^{(1)}(\az,\lz)1_{L_n}(\Delta z_k)+\bar{q}^{(2)}_{h_n}(z_{k-1},z_k;\al_0,\lam_0)\bigg]dz_k\pi_{k-1}(dz_{k-1})\\
    \nn&=o\Big((n\log h_n)^{-1}\Big).
\end{align}
Corollary \ref{testi} implies that for any $z_{k-1}\in\mbbr$, we have
\begin{align*}
    &\bigg|\log q_k(\al_0,\lam_0)\bigg|\\
    &= -\log q_k(\al_0,\lam_0) 1_{q_k(\al_0,\lam_0)<1} + \log q_k(\al_0,\lam_0)1_{q_k(\al_0,\lam_0)\geq1}\\
    &\leq - \log q_k(\al_0,\lam_0)1_{q_k(\al_0,\lam_0)<1} + \log \Bigg(C\left[\frac{1}{h_n^{d/2}}\exp\left(-\frac{|\D z_k|^2}{Ch_n}\right)+h_n\exp\left(-u\zeta||\D z_k||_1\right)\right]\Bigg)1_{q_k(\al_0,\lam_0)\geq1}\\
    &\leq - \log q_k(\al_0,\lam_0)1_{q_k(\al_0,\lam_0)<1}+\log C-\frac{d}{2}\log h_n,
\end{align*}
for large enough $n$.
Combined with \eqref{tvd}, we get
\begin{align*}
    &\Bigg|\ezn[l_n(\mbfz_n,\al_0,\lam_0)]-\sumk\intd\intd\log q_k(\al_0,\lam_0) \tilde{q}_k(\al_0,\lam_0)dz_k \pi_{k-1}(dz_{k-1})\Bigg|\\
    &\leq \sumk\intd\intd\bigg|\log q_k(\al_0,\lam_0)\bigg|(q_{k}(\al_0,\lam_0)-\tilde{q}_k(\al_0,\lam_0))dz_k \pi_{k-1}(dz_{k-1})\\
    &=\sumk\intd\intd- \log q_k(\al_0,\lam_0)1_{q_k(\al_0,\lam_0)<1}(q_{k}(\al_0,\lam_0)-\tilde{q}_k(\al_0,\lam_0))dz_k \pi_{k-1}(dz_{k-1})+o(1)\\
    &=:\sumk\intd \bigg(G_0(z_{k-1})+G_1(z_{k-1})+G_2(z_{k-1})\bigg)\pi_{k-1}(dz_{k-1})+o(1),
\end{align*}
where
\begin{align*}
    &G_0(z_{k-1})=e^{-\lam_0 h_n}\intd- \log q_k(\al_0,\lam_0)1_{q_k(\al_0,\lam_0)<1}p_k(\az)1_{L_n^c}(\D z_k)dz_k,\\
    &G_1(z_{k-1})=e^{-\lam_0 h_n}\intd- \log q_k(\al_0,\lam_0)1_{q_k(\al_0,\lam_0)<1}q_k^{(1)}(\az,\lam_0)1_{L_n}(\D z_k)dz_k,\\
    &G_2(z_{k-1})=e^{-\lam_0 h_n}\intd- \log q_k(\al_0,\lam_0)1_{q_k(\al_0,\lam_0)<1}\bar{q}^{(2)}_{h_n}(z_{k-1},z_k;\al_0,\lam_0)dz_k.
\end{align*}
By Aronson-type estimate, we obtain
\begin{align*}
    &G_0(z_{k-1})\\
    &\leq \int_{|\D x_k|>h_n^\rho} -\log p_k(\az)1_{q_k(\al_0,\lam_0)<1}  p_k(\az)dz_k\\
    &\lesssim \int_{|\D z_k|>h_n^\rho}\Bigg(C-\frac{d}{2}\log h_n+\frac{|\D z_k|^2}{h_n}\Bigg)h_n^{-d/2}\exp\Bigg[-\frac{|\D z_k|^2 }{Ch_n}\Bigg]dz_k\\
    &\lesssim \int_{|v|>h_n^{\rho-\frac{1}{2}}}\Bigg(C-\frac{d}{2}\log h_n+v^2\Bigg)\exp\Bigg[-\frac{|v|^2}{C}\Bigg]dv\\
    &=o(n^{-1}).
\end{align*}
Since for any $\kappa\in(0,1)$, there exists a positive constant $C_\kappa$ such that for any $x\in(0,1)$,
\begin{equation}\label{mlogineq}
    -\log x\leq x^{-\kappa}+C_{\kappa},
\end{equation}
it follows from H\"{o}lder's inequality that for small enough $\kappa>0$,
\begin{align*}
    &\intd G_1(z_{k-1})\pi_{k-1}(dz_{k-1})\\
    &\leq\intd\intd-\log q_k^{(1)}(\al_0,\lam_0)dz_k1_{q_k^{(1)}(\al_0,\lam_0)<1}q_k^{(1)}(\az,\lam_0)1_{L_n}(\D z_k)dz_k\pi_{k-1}(dz_{k-1})\\
    &\leq\intd\Bigg(\intd  q_k^{(1)}(\al_0,\lam_0)1_{L_n}(\D z_k)dz_k\Bigg)^{1-\kappa}\Bigg(\intd 1_{L_n}(\D z_k)dz_k\Bigg)^{\kappa}\pi_{k-1}(z_{k-1})+o(n^{-1})\\
    &\lesssim h_n^{(1+(d+\gam)\rho\wedge1)(1-\kappa)+d\kappa}+o(n^{-1})\\
    &=o(n^{-1}).
\end{align*}
In the above estimate, we used the fact that the volume of the $d$-dimensional sphere with radius $r$ is given by $\pi^{d/2}r^d/\Gam(d/2+1)$.
By utilizing \eqref{mlogineq} and Corollary \ref{testi}, we also have
\begin{align*}
    &G_2(z_{k-1})\\
    &\leq\intd- \log \bar{q}^{(2)}_{h_n}(z_{k-1},z_k;\al_0,\lam_0)1_{\bar{q}^{(2)}_{h_n}(z_{k-1},z_k;\al_0,\lam_0)<1}\bar{q}^{(2)}_{h_n}(z_{k-1},z_k;\al_0,\lam_0)dz_k\\
    &\lesssim h_n^{1+(m+\gam)\rho\wedge1}\intd\exp\big(-C||\D z_k||_1\big) dz_k+C_\kappa h_n^2\\
    &=o(n^{-1}),
\end{align*}
for small enough $\kappa>0$.
Hence, the desired result follows.
\end{proof}

From the elementary inequality $\log(1+x)\leq 1+x$ for any $x\geq0$ and \eqref{tvd}, we have
\begin{align*}
    &\sumk\intd\intd\Bigg[\log q_k(\az,\lz)-\log\tilde{q}_k(\az,\lz)\Bigg] \tilde{q}_k(\al_0,\lam_0)dz_k \pi_{k-1}(dz_{k-1})\\
    &=\sumk\intd\intd\log\bigg(1+\frac{q_k(\az,\lz)-\tilde{q}_k(\az,\lz)}{\tilde{q}_k(\az,\lz)}\bigg)\tilde{q}_k(\al_0,\lam_0)dz_k \pi_{k-1}(dz_{k-1})\\
    &\leq \sumk\intd\intd \bigg[q_k(\az,\lz)-\tilde{q}_k(\az,\lz)\bigg]dz_k \pi_{k-1}(dz_{k-1})\\
    &=o(1).
\end{align*}
Hence we obtain
\begin{align}
    &\label{estz}\sumk\intd\intd\log q_k(\az,\lz) \tilde{q}_k(\al_0,\lam_0)dz_k \pi_{k-1}(dz_{k-1})\\
    \nn &=\sumk\intd\intd\log \tilde{q}_k(\az,\lz) \tilde{q}_k(\al_0,\lam_0)dz_k \pi_{k-1}(dz_{k-1})+o(1),
    % &\label{estest}\sumk\intd\intd\log q_k(\aes,\les) \tilde{q}_k(\al_0,\lam_0)dx_k \pi_{k-1}(dx_{k-1})\\
    % \nn &=\sumk\intd\intd\log \tilde{q}_k(\aes,\les) \tilde{q}_k(\al_0,\lam_0)dx_k \pi_{k-1}(dx_{k-1})+o(1).
\end{align}
and combined with Lemma \ref{lemapp1}, we obtain \eqref{app1}.

Next, we introduce some lemmas for \eqref{app2}.

\begin{Lem}\label{momratio}
    Suppose that Assumptions \ref{inidist}-\ref{indentifiability} hold. Then there exists a positive constant $C_A$ such that for any $k,l\in\mbbn$, $\al\in\Theta_\al$ and $z_{k-1}\in\mbbr^d$, 
    $$\intd \Bigg(\frac{q^{(l)}_k(\az,\lz)}{q^{(l)}_k(\al,\lam)}\Bigg)^{1+(\xi\wedge C_A)}q^{(l)}_k(\al,\lam)dz_k\lesssim h_n^l$$
\end{Lem}

\begin{proof}
Let $l=1$.
For simplicity, we write
$$g_k(u_1,u_2,\al,\lam)=\lam p_{\tau_1}(z_{k-1},u_1;\al) F_{\theta}(u_2) p_{h_n-\tau_1}(u_1+u_2, z_k; \al).$$
From Jensen's inequality, for all $m>1$, we have
    \begin{align*}
        &\intd \Bigg(\frac{q^{(l)}_k(\az,\lz)}{q^{(l)}_k(\al,\lam)}\Bigg)^{m}q^{(l)}_k(\al,\lam)dz_k\\
        &=\intd \Bigg(\int_0^{h_n}\intd\intd \frac{g_k(u_1,u_2,\az,\lz)}{g_k(u_1,u_2,\al,\lam)} \frac{g_k(u_1,u_2,\al,\lam)}{q^{(1)}_k(\al,\lam)}du_1du_2d\tau_1\Bigg)^{m}q^{(1)}_k(\al,\lam)dz_k\\
        &\leq \intd\int_0^{h_n}\intd\intd \Bigg(\frac{g_k(u_1,u_2,\az,\lz)}{g_k(u_1,u_2,\al,\lam)} \Bigg)^{m} g_k(u_1,u_2,\al,\lam)du_1du_2d\tau_1dz_k.
    \end{align*}
Hence from Aronson-type estimate and Assumption \ref{jdist}-(1), there exists a positive constant $C_A$ such that
\begin{align*}
    &\intd\int_0^{h_n}\intd\intd \Bigg(\frac{g_k(u_1,u_2,\az,\lz)}{g_k(u_1,u_2,\al,\lam)} \Bigg)^{1+(\xi\wedge C_A)} g_k(u_1,u_2,\al,\lam)du_1du_2d\tau_1dz_k\\
    &\lesssim \intd\int_0^{h_n}\intd \frac{F_{\tz}(u_2)}{(F_\theta(u_2))^{\xi\wedge C_A}}h^{-\frac{d}{2}}\exp\Bigg(-\frac{|u_2-(z_k-z_{k-1})|^2}{Ch_n}\Bigg)du_2d\tau_1dz_k\\
    &\lesssim h_n\intd\frac{F_{\tz}(u_2)}{(F_\theta(u_2))^{\xi\wedge C_A}}du_2\lesssim h_n.
\end{align*}
For general $l$, the proof is similar and thus we omit it.
\end{proof}

\begin{Lem}\label{derpar}
    Under Assumptions \ref{inidist}-\ref{indentifiability}, there exists a positive constant $C$ such that for any $k\in\mbbn$, $l\in\{1,2\}$, $m_1\in\{1,2,3,4,5\}$, $m_2\geq1$, $\al\in\Theta_\al$, and $z_{k-1}\in\mbbr^d$,
    $$\intd \Bigg|\frac{\p_\al^{m_1} q_k^{(l)}(\al,\lam)}{q_k^{(l)}(\al,\lam)}\Bigg|^{m_2}q_k^{(l)}(\al_0,\lz)dz_k\lesssim h_n^{l}(1+|z_{k-1}|^C).$$
\end{Lem}

\begin{proof}
    Let $d_k^{(l)}(\al,\lam)=\intd q_k^{(l)}(\al,\lam)dz_k$ and $\kappa=1+(\xi\wedge C_A)$.
    By applying Jensen's inequality and H\"{o}lder's inequality, it follows from Lemma \ref{momratio} and the argument in the proof of \cite[Proposition 5.4]{OgiUeh23} that
    \begin{align*}
        &\intd \Bigg|\frac{\p_\al^{m_1} q_k^{(l)}(\al,\lam)}{q_k^{(l)}(\al,\lam)}\Bigg|^{m_2}q_k^{(l)}(\al_0,\lz)dz_k\\
        &=\intd \Bigg|\frac{\p_\al^{m_1} q_k^{(l)}(\al,\lam)}{q_k^{(l)}(\al,\lam)}\Bigg|^{m_2}\frac{q_k^{(l)}(\al_0,\lz)}{q_k^{(l)}(\al,\lam)} \frac{q_k^{(l)}(\al,\lam)}{d_k^{(l)}(\al,\lam)}dz_kd_k^{(l)}(\al,\lam)\\
        &\leq \Bigg(\intd \Bigg|\frac{\p_\al^{m_1} q_k^{(l)}(\al,\lam)}{q_k^{(l)}(\al,\lam)}\Bigg|^{\frac{\kappa m_2}{\kappa-1}}\frac{q_k^{(l)}(\al,\lam)}{d_k^{(l)}(\al,\lam)}dz_k\Bigg)^{\frac{\kappa-1}{\kappa}}\Bigg(\intd \Bigg(\frac{q^{(l)}_k(\az,\lz)}{q^{(l)}_k(\al,\lam)}\Bigg)^{\kappa}\frac{q_k^{(l)}(\al,\lam)}{d_k^{(l)}(\al,\lam)}dz_k\Bigg)^{\frac{1}{\kappa}}d_k^{(l)}(\al,\lam)\\
        &=\Bigg(\intd \Bigg|\frac{\p_\al^{m_1} q_k^{(l)}(\al,\lam)}{q_k^{(l)}(\al,\lam)}\Bigg|^{\frac{\kappa m_2}{\kappa-1}}q_k^{(l)}(\al,\lam)dz_k\Bigg)^{\frac{\kappa-1}{\kappa}}\Bigg(\intd \Bigg(\frac{q^{(l)}_k(\az,\lz)}{q^{(l)}_k(\al,\lam)}\Bigg)^{\kappa}q_k^{(l)}(\al,\lam)dz_k\Bigg)^{\frac{1}{\kappa}}\\
        &\lesssim h_n^l(1+|z_{k-1}|^C).
    \end{align*}
    
\end{proof}

Similarly to the proof of Lemma \ref{lemapp1}, we have
\begin{align*}
    &E\Bigg[\Bigg|\ezn[l_n(\mbfz_n,\aes,\les)]-\sumk\intd\intd\log q_k(\aes,\les) \tilde{q}_k(\al_0,\lam_0)dz_k \pi_{k-1}(dz_{k-1})\Bigg|\Bigg]\\
    &\leq E\Bigg[\Bigg|\sumk\intd \bigg(\hat{G}_1(z_{k-1})+\hat{G}_2(z_{k-1})+\hat{G}_3(z_{k-1})\bigg)\pi_{k-1}(dz_{k-1})\Bigg|\Bigg]+o(1),
\end{align*}
where
\begin{align*}
    &\hat{G}_1(z_{k-1})=e^{-\lam h_n}\intd- \log q_k(\aes,\les)1_{q_k(\aes,\les)<1}q_k^{(1)}(\az,\lam_0)1_{L_n}(\D z_k)dz_k,\\
    &\hat{G}_2(z_{k-1})=e^{-\lam h_n}\intd- \log q_k(\aes,\les)1_{q_k(\aes,\les)<1}q^{(2)}_{h_n}(z_{k-1},z_k;\al_0,\lam_0)dz_k,\\
    &\hat{G}_3(z_{k-1})=e^{-\lam h_n}\intd- \log q_k(\aes,\les)1_{q_k(\aes,\les)<1}\bar{q}^{(3)}_{h_n}(z_{k-1},z_k;\al_0,\lam_0)dz_k.
\end{align*}

By using \eqref{mlogineq}, for any $\kappa\in(0,1)$, there exists a positive constant $C_\kappa$ such that
$$|\log x|\leq x^{-\kappa}+C_\kappa+x,$$
for all $x\in\mbbr$.
Moreover, Lemma \ref{ktesti} implies that $q_k^{(1)}(\az,\lam_0)$ is upper-bounded.
Hence, for small enough $\kappa>0$, it follows from Taylor expansion, Cauchy-Schwartz inequality, and Lemma \ref{derpar} that
\begin{align*}
    &\hat{G}_1(z_{k-1})\\
    &\leq \intd -\log q_k^{(1)}(\aes,\les)1_{q_k(\aes,\les)<1}q_k^{(1)}(\az,\lam_0)1_{L_n}(\D z_k)dz_k\\
    &\leq \intd \big|\log q_k^{(1)}(\az,\lz)\big|1_{q_k(\aes,\les)<1}q_k^{(1)}(\az,\lam_0)1_{L_n}(\D z_k)dz_k\\
    &+ \intd \Bigg|\frac{\p_\al q_k^{(1)}(\check{\al}_n,\check{\lam}_n)}{q_k^{(1)}(\check{\al}_n,\check{\lam}_n)}\Bigg|q_k^{(1)}(\az,\lam_0) 1_{L_n}(\D z_k)dz_k\big|\aes-\az\big|\\
    &\leq \intd \Big[\big(q_k^{(1)}(\az,\lam_0)\big)^{-\kappa} + q_k^{(1)}(\az,\lam_0) + C_\kappa\Big]q_k^{(1)}(\az,\lam_0)1_{L_n}(\D z_k)dz_k\\
    &+\Bigg(\intd \Bigg|\frac{\p_\al q_k^{(1)}(\check{\al}_n,\check{\lam}_n)}{q_k^{(1)}(\check{\al}_n,\check{\lam}_n)}\Bigg|^{1+\frac{1}{\eta}}q_k^{(1)}(\az,\lam_0)dz_k\Bigg)^{\frac{\eta}{1+\eta}} \Bigg(\intd q_k^{(1)}(\az,\lam_0) 1_{L_n}(\D z_k)dz_k\Bigg)^{\frac{1}{1+\eta}}\big|\aes-\az\big|\\
    &\lesssim \bigg[h_n^{(1+(d+\gam)\rho\wedge1)(1-\kappa)+d\kappa}+h_n^{1+\frac{(d+\gam)\rho\wedge1}{1+\eta}}\big|\aes-\az\big|\bigg]\big(1+|z_{k-1}|^C\big)+o(n^{-1}),
\end{align*}
with $\check{\al}_n=s\az+(1-s)\aes$ for $s\in(0,1)$. 
In the last inequality, we have used the same estimates for $G_1(z_{k-1})$ appearing in the proof of Lemma \ref{lemapp1}.
Similarly, Lemma \ref{derpar} yields that
\begin{align*}
    &\hat{G}_2(z_{k-1})\\
    &\leq \intd \Bigg|\frac{\p_\al q_k^{(2)}(\check{\al}_n,\check{\lam}_n)}{q_k^{(2)}(\check{\al}_n,\check{\lam}_n)}\Bigg|q_k^{(2)}(\az,\lam_0)dz_k\big|\aes-\az\big|+o(n^{-1})\\
    &\lesssim h_n^2(1+|z_{k-1}|^C)\big|\aes-\az\big|+o(n^{-1}).
\end{align*}
By making use of Aronson-type estimate and Corollary \ref{testi}, we have
\begin{align*}
    &\hat{G}_3(z_{k-1})\\
    &\leq \intd -\log p_k(\aes)1_{q_k(\aes,\les)<1} \bar{q}^{(3)}_{h_n}(z_{k-1},z_k;\al_0,\lam_0)dz_k\\
    &\lesssim h_n^3\intd\Bigg(C-\frac{d}{2}\log h_n+\frac{|\D z_k|^2}{h_n}\Bigg)\exp\Big(-u\zeta|\D z_k|\Big)dz_k\\
    &\lesssim h_n^2.
\end{align*}
Combined with the moment convergence of $|\aes-\az|$, we get 
$$E\Bigg[\ezn[l_n(\mbfz_n,\aes,\les)]-\sumk\intd\intd\log q_k(\aes,\les) \tilde{q}_k(\al_0,\lam_0)dz_k \pi_{k-1}(dz_{k-1})\Bigg]=o(1).$$
Hence, for \eqref{app2}, it remains to show 
\begin{equation}\label{negli}
    E\Bigg[\sumk\intd\intd\big[\log q_k(\aes,\les)-\log \tilde{q}_k(\aes,\les)\big]\tilde{q}_k(\al_0,\lam_0)dz_k \pi_{k-1}(dz_{k-1})\Bigg]=o(1).
\end{equation}
Let $\hat{v}_{n,\al}=\aes-\az$.
By the definition of $\tilde{q}_k$ and the fundamental inequality: $\log(1+x+y)\leq \log(1+x)+\log(1+y)$ for any $x,y\geq0$, we obtain the following.
\begin{align*}
    0&\leq\intd\Bigg[\log q_k(\aes,\les)-\log\tilde{q}_k(\aes,\les)\Bigg] \tilde{q}_k(\al_0,\lam_0)dz_k\\
    &=\intd \log \Bigg(1+\frac{q_k^{(1)}(\aes,\les)+\bar{q}_k^{(2)}(\aes,\les)}{p_k(\aes)}\Bigg)p_k(\az)1_{L_n}(\D z_k)dz_k\\
    &+\intd \log \Bigg(1+\frac{p_k(\aes)+\bar{q}_k^{(2)}(\aes,\les)}{q_k^{(1)}(\aes,\les)}\Bigg)q_k^{(1)}(\az,\lz)1_{L_n^c}(\D z_k)dz_k\\
    &\leq \intd \log \Bigg(1+\frac{q_k^{(1)}(\aes,\les)}{p_k(\aes)}\Bigg)p_k(\az)1_{L_n}(\D z_k)dz_k\\
    &+\intd \log \Bigg(1+\frac{p_k(\aes)}{q_k^{(1)}(\aes,\les)}\Bigg)q_k^{(1)}(\az,\lz)1_{L_n^c}(\D z_k)dz_k\\
    &+\intd \log \Bigg(1+\frac{\bar{q}_k^{(2)}(\aes,\les)}{p_k(\aes)}\Bigg)p_k(\az)1_{L_n}(\D z_k)dz_k\\
    &+\intd \log \Bigg(1+\frac{\bar{q}_k^{(2)}(\aes,\les)}{q_k^{(1)}(\aes,\les)}\Bigg)q_k^{(1)}(\az,\lz)1_{L_n^c}(\D z_k)dz_k.
\end{align*}
We will derive the estimates of the last quantities in the following lemma.

\begin{Lem}\label{lem:negli}
    Under Assumptions \ref{inidist}-\ref{indentifiability}, it follows that
    \begin{equation*}
        E\Bigg[\sumk\intd\intd \log \Bigg(1+\frac{q_k^{(1)}(\aes,\les)}{p_k(\aes)}\Bigg)p_k(\az)1_{L_n}(\D z_k)dz_k\pi_{k-1}(dz_{k-1})\Bigg]= o(1).
    \end{equation*}
\end{Lem}
\begin{proof}
    From Taylor expansion, we have
    \begin{align}
        \nn&\intd \log \Bigg(1+\frac{q_k^{(1)}(\aes,\les)}{p_k(\aes)}\Bigg)p_k(\az)1_{L_n}(\D z_k)dz_k\\
        \nn&\leq \intd \log \Bigg(1+\frac{\bar{\lam}q_k^{(1)}(\aes)}{p_k(\aes)}\Bigg)p_k(\az)1_{L_n}(\D z_k)dz_k\\
        \nn&=R_{0,n,k-1}+R_{1,n,k-1}+R_{2,n,k-1},
    \end{align}
    where $R_{0,n,k-1}$, $R_{1,n,k-1}$, and $R_{2,n,k-1}$ are defined as
    \begin{align}
        \nn&R_{0,n,k-1}=\intd \log \Bigg(1+\frac{\bar{\lam}q_k^{(1)}(\az)}{p_k(\az)}\Bigg)p_k(\az)1_{L_n}(\D z_k)dz_k,\\
        % \nn&R_{1,n,k-1}=\bar{\lam}\Bigg[\intd \frac{\p_\al q_k^{(1)}(\az)p_k(\az)-\p_\al p_k(\az) q_k^{(1)}(\az)}{p_k(\az)+\bar{\lam} q_k^{(1)}(\az)}1_{L_n}(\D x_k)dx_k\Bigg]\bigg[\hat{v}_{n,\al}\bigg],\\
        \nn&R_{1,n,k-1}=\sum_{l=1}^4\frac{1}{l!}\Bigg[\intd\p_\al^{l}\log \Bigg(1+\frac{\bar{\lam}q_k^{(1)}(\az)}{p_k(\az)}\Bigg)p_k(\az)1_{L_n}(\D z_k)dz_k\Bigg]\bigg[\hat{v}_{n,\al}^{\otimes l}\bigg],\\
        % \nn& R_{2,n,k-1}=\Bigg[\intd \Bigg\{\frac{\p_\al^{2}p_k(\az)+\bar{\lam} \p_\al^{2}q_k^{(1)}(\az)}{p_k(\az)+\bar{\lam} q_k^{(1)}(\az)}-\frac{(\p_\al p_k(\az)+\bar{\lam} \p_\al q_k^{(1)}(\az))^2}{(p_k(\az)+\bar{\lam} q_k^{(1)}(\az))^2}\\
        % \nn &\qquad\qquad\qquad -\frac{\p_\al^{2}p_k(\az)}{p_k(\az)}+\frac{(\p_\al p_k(\az))^2}{p^2_k(\az)}\Bigg\}p_k(\az)1_{L_n}(\D x_k)dx_k\Bigg]\bigg[\hat{v}_{n,\al}^{\otimes2}\bigg],\\
        \nn& R_{2,n,k-1}=\frac{1}{5!}\Bigg[\intd\p_\al^5\log \Bigg(1+\frac{\bar{\lam}q_k^{(1)}(\check{\al}_n)}{p_k(\check{\al}_n)}\Bigg)p_k(\az)1_{L_n}(\D z_k)dz_k\Bigg]\bigg[\hat{v}_{n,\al}^{\otimes5}\bigg],
    \end{align}
    with $\check{\al}_n=s\az+(1-s)\aes$ for $s\in(0,1)$. 
    In the following, we look at these terms separately.
    From the proof of \eqref{estz}, we have
    $$0\leq R_{0,n,k-1}=\intd \log \Bigg(1+\frac{\bar{\lam}q_k^{(1)}(\az)}{p_k(\az)}\Bigg)p_k(\az)1_{L_n}(\D z_k)dz_k\lesssim h_n^{1+((d+\gamma)\rho)\wedge 1}\bigg(1+|z_{k-1}|^C\bigg),$$
    for any $z_{k-1}\in\mbbr^d$, and thus we obtain
    $$E\Bigg[\sumk\intd R_{0,n,k-1}\pi_{k-1}(dz_{k-1})\Bigg]=o(1).$$
    Next we observe the first term of $R_{1,n,k-1}$.
    We consider the following decomposition:
    \begin{align}
        \nn& \p_\al\log \Bigg(1+\frac{\bar{\lam}q_k^{(1)}(\az)}{p_k(\az)}\Bigg)p_k(\az)\\
        \nn&=\frac{\p_\al q_k^{(1)}(\az)p_k(\az)-\p_\al p_k(\az) q_k^{(1)}(\az)}{p_k(\az)+\bar{\lam} q_k^{(1)}(\az)}\bar{\lam}\\
        \nn&=\p_\al q_k^{(1)}(\az)\bar{\lam}-\frac{\bar{\lam}q_k^{(1)}(\az)}{p_k(\az)+\bar{\lam} q_k^{(1)}(\az)}\bigg[\p_\al p_k(\az)+\bar{\lam}\p_\al q_k^{(1)}(\az)\bigg].
    \end{align}
    \eqref{pj1c},  H\"{o}lder's inequality, and Lemma \ref{derpar} yield that
    \begin{align}
    \nn &\Bigg|\intd\intd \p_\al q_k^{(1)}(\az)  1_{L_n}(\D z_k)dz_k \pi_{k-1}(dz_{k-1})\Bigg|\\
    \nn &\leq\intd\Bigg(\intd \Bigg|\frac{\p_\al q_k^{(1)}(\az)}{q_k^{(1)}(\az)}\Bigg|^{1+\frac{1}{\eta}} q_k^{(1)}(\az) dz_k\Bigg)^{\frac{\eta}{1+\eta}}\Bigg(\intd q_k^{(1)}(\az)1_{L_n}(\D z_k)dz_k\Bigg)^{\frac{1}{1+\eta}} \pi_{k-1}(dz_{k-1})\\
    \label{dpj1c} &\lesssim h_n^{1+\frac{(d+\gam)\rho\wedge1}{1+\eta}}.
    \end{align}
    Since $\frac{\bar{\lam}q_k^{(1)}(\az)}{p_k(\az)+\bar{\lam} q_k^{(1)}(\az)}<1$,
    H\"{o}lder's inequality and Jensen's inequality: $(a+b)^{1+\ep}\leq 2^{1+\ep}(a^{1+\ep}+b^{1+\ep})$ for any $\ep>0$ and $a,b>0$ lead to
    \begin{align}
        \label{estprop}&\Bigg|\intd \frac{\p_\theta q_k^{(1)}(\az)p_k(\az)-\p_\theta p_k(\az) q_k^{(1)}(\az)}{p_k(\az)+\bar{\lam} q_k^{(1)}(\az)}1_{L_n}(\D z_k)dz_k\Bigg|\\
        \nn&\lesssim \Bigg|\intd \p_\theta q_k^{(1)}(\az)1_{L_n}(\D z_k)dz_k\Bigg|+\intd \Bigg(\frac{q_k^{(1)}(\az)}{p_k(\az)+\bar{\lam} q_k^{(1)}(\az)}\Bigg)^{\frac{1}{1+\eta}}|\p_\theta p_k(\az)+\bar{\lam}\p_\theta q_k^{(1)}(\az)|1_{L_n}(\D z_k)dz_k\\
        \nn&\lesssim \Bigg|\intd \p_\theta q_k^{(1)}(\az)1_{L_n}(\D z_k)dz_k\Bigg|\\
        \nn&+\Bigg(\intd q_k^{(1)}(\az)1_{L_n}(\D z_k)dz_k\Bigg)^{\frac{1}{1+\eta}}\Bigg(\intd \Bigg[\Bigg|\frac{\p_\theta p_k(\az)}{p_k(\az)}\Bigg|^{1+\frac{1}{\eta}}p_k(\az)+\Bigg|\frac{\p_\theta q_k^{(1)}(\az)}{q_k^{(1)}(\az)}\Bigg|^{1+\frac{1}{\eta}}q_k^{(1)}(\az)\Bigg]dz_k\Bigg)^{\frac{\eta}{1+\eta}}\\
        \nn&\lesssim h_n^{1 +\frac{((d+\gamma)\rho)\wedge 1}{1+\eta}}\bigg(1+|z_{k-1}|^C\bigg).
    \end{align}
    In the last inequality, we used \eqref{pj1c}, \cite[Proposition A.2]{OgiUeh23}, and Lemma \ref{derpar}.
    Similarly, we also have
    \begin{equation}
        \nn \Bigg|\intd \frac{\p_\sig q_k^{(1)}(\az)p_k(\az)-\p_\sig p_k(\az) q_k^{(1)}(\az)}{p_k(\az)+\bar{\lam} q_k^{(1)}(\az)}1_{L_n}(\D z_k)dz_k\Bigg|\lesssim h_n^{\frac{1+((d+\gamma)\rho)\wedge 1}{1+\eta}}\bigg(1+|z_{k-1}|^C\bigg).
    \end{equation}
    % In the same way, we can also evaluate the rest of $R_{1,n,k-1}$.
    Hence, the moment convergence of the QMLE and Assumption \ref{balance-cond} implies that 
    \begin{align}
       &\nn E\Bigg[\sumk\intd \Bigg|\intd\p_\al\log \Bigg(1+\frac{\bar{\lam}q_k^{(1)}(\az)}{p_k(\az)}\Bigg)p_k(\az)1_{L_n}(\D z_k)dz_k\Bigg]\bigg[\hat{v}_{n,\al}\bigg]\Bigg|\pi_{k-1}(dz_{k-1})\Bigg]\\
       &\label{neg1}\lesssim \sqrt{n}h_n^{(\frac{1}{1+\eta}\vee \frac{1}{2})+\frac{((d+\gamma)\rho)\wedge 1}{1+\eta}}=o(1).
    \end{align}
    The term 
    $$\p_\al^2\log \Bigg(1+\frac{\bar{\lam}q_k^{(1)}(\az)}{p_k(\az)}\Bigg)p_k(\az)=\Big(\p_\al^2\log (p_k(\az)+\bar{\lam}q_k^{(1)}(\az))-\p_\al \log p_k(\az)\Big)p_k(\az)$$
    can be decomposed as
    \begin{align*}
        &\Big(\p_\al^2\log (p_k(\az)+\bar{\lam}q_k^{(1)}(\az))-\p_\al \log p_k(\az)\Big)p_k(\az)=\mathcal{J}_{1,k}+\mathcal{J}_{2,k},
    \end{align*}
    where
    \begin{align*}
    \mathcal{J}_{1,k}&=\bar{\lam}\p_\al^2q_k^{(1)}(\az)-\frac{\bar{\lam}q_k^{(1)}(\az)}{p_k(\az)+\bar{\lam} q_k^{(1)}(\az)}\big[\p_\al^2 p_k(\az)+\p_\al^2q_k^{(1)}(\az)\big],\\
        \mathcal{J}_{2,k}&=\frac{2\bar{\lam}(\p_\al p_k(\az))^2}{(p_k(\az)+q_k^{(1)}(\az))^2}q_k^{(1)}(\az)+\frac{\bar{\lam}^2(\p_\al p_k(\az))^2}{(p_k(\az)+q_k^{(1)}(\az))^2p_k(\az)}(q_k^{(1)}(\az))^2\\
        &-\frac{2\bar{\lam}\p_\al p_k(\az)\p_\al q_k^{(1)}(\az)}{p_k(\az)+q_k^{(1)}(\az)}+\frac{2\bar{\lam}\p_\al p_k(\az)\p_\al q_k^{(1)}(\az)}{(p_k(\az)+q_k^{(1)}(\az))^2}q_k^{(1)}(\az)\\
        &+\frac{\bar{\lam}^2(\p_\al q_k^{(1)}(\az))^2}{p_k(\az)+q_k^{(1)}(\az)}-\frac{\bar{\lam}^2(\p_\al q_k^{(1)}(\az))^2}{(p_k(\az)+q_k^{(1)}(\az))^2}q_k^{(1)}(\az).
    \end{align*}
    As in the proof of \eqref{neg1}, it is straightforward that
    $$E\Bigg[\sumk\intd\Bigg|\intd \mathcal{J}_{1,k}1_{L_n}(\D z_k)dz_k[\hat{v}_{n,\al}^{\otimes 2}]\Bigg|\pi_{k-1}(dz_{k-1})\Bigg]=o(1).$$
    As for $\mathcal{J}_{2,k}$, it follows  from \eqref{pj1c}, \cite[Proposition A.2]{OgiUeh23} and H\"{o}lder's inequality that
    \begin{align*}
        &\intd \frac{(\p_\theta p_k(\az))^2}{(p_k(\az)+q_k^{(1)}(\az))^2}q_k^{(1)}(\az)dz_{k}\\
        &=\intd \frac{(\p_\theta p_k(\az))^2}{(p_k(\az)+q_k^{(1)}(\az))^{2-\frac{\eta}{1+\eta}}}\Big(q_k^{(1)}(\az)\Big)^{\frac{1}{1+\eta}}dz_{k}\\
        &\leq \Bigg(\intd q_k^{(1)}(\az)dz_{k}\Bigg)^{\frac{1}{1+\eta}}\Bigg(\intd \Bigg|\frac{\p_\theta p_k(\az)}{p_k(\az)}\Bigg|^{\frac{2(1+\eta)}{\eta}}p_k(\az)dz_k\Bigg)^{\frac{\eta}{1+\eta}}\\
        &\lesssim h_n^{1+\frac{(d+\gamma)\rho)\wedge 1}{1+\eta}}(1+|z_{k-1}|^C).
    \end{align*}
    The rest term of $\mathcal{J}_{2,k}$ can be evaluated in a similar manner, and thus we get 
    $$E\Bigg[\sumk\intd\Bigg|\intd \mathcal{J}_{2,k}1_{L_n}(\D z_k)dz_k[\hat{v}_{n,\al}^{\otimes 2}]\Bigg|\pi_{k-1}(dz_{k-1})\Bigg]=o(1).$$
    Regarding the estimates of $l=3,4$ in the summand of $R_{1,n,k-1}$, its proof is almost the same and we omit the detail.
    Therefore, we arrive at
    \begin{equation}
        \nn E\Bigg[\sumk\intd \big|R_{1,n,k-1}\big|\pi_{k-1}(dz_{k-1})\Bigg]=o(1).
    \end{equation}
    
    Finally, we check
    \begin{equation}\label{reminder2}
         E\Bigg[\sumk\intd \big|R_{2,n,k-1}\big|\pi_{k-1}(dz_{k-1})\Bigg]=o(1).
    \end{equation}
    Now we show that there exist some positive constants $\ep$ and $C$ such that
    \begin{equation}\label{reminder5}
     \Bigg|\intd\p_\sig^i\p_\theta^{5-i}\log \Bigg(1+\frac{\bar{\lam}q_k^{(1)}(\check{\al}_n)}{p_k(\check{\al}_n)}\Bigg)p_k(\az)1_{L_n}(\D z_k)dz_k\Bigg|\lesssim h_n^{\ep 1_{i\neq 5}} (1+|z_{k-1}|^C),
    \end{equation}
    for all $i\in\{0,1,2,3,4,5\}$.
    To avoid redundancy, we only focus on $$\Bigg|\intd \frac{\p^5_\theta q_k^{(1)}(\check{\al})}{p_k(\check{\al}_n)+\bar{\lam}q_k^{(1)}(\check{\al}_n)}p_k(\az)1_{L_n}(\D z_k)dz_k\Bigg|,$$
    which appears in the left-hand-side of \eqref{reminder5}.
    From Lemma \ref{derpar}, Aronson-type estimates, and H\"{o}lder's inequality, there exists positive constant $C>1$ such that
    % and Sobolev's inequality, it follows that for any $\mu >p_\al$,
    % \begin{align}
    %     \nn&E\Bigg[\Bigg|\intd \frac{\p^5_\theta q_k^{(1)}(\check{\al}_n)}{p_k(\check{\al}_n)+\bar{\lam}q_k^{(1)}(\check{\al}_n)}p_k(\az)1_{L_n}(\D z_k)dz_k\Bigg|^\mu\Bigg]\\
    %     \nn&\leq E\Bigg[\sup_\al\Bigg|\intd \frac{\p^5_\theta q_k^{(1)}(\al)}{p_k(\al)+\bar{\lam}q_k^{(1)}(\al)}p_k(\az)1_{L_n}(\D z_k)dz_k\Bigg|^\mu\Bigg]\\
    %     \nn&\lesssim \sup_\al\Bigg\{E\Bigg[\Bigg|\intd \frac{\p^5_\theta q_k^{(1)}(\al)}{p_k(\al)+\bar{\lam}q_k^{(1)}(\al)}p_k(\az)1_{L_n}(\D z_k)dz_k\Bigg|^\mu\Bigg]+E\Bigg[\Bigg|\intd \frac{\p^6_\theta q_k^{(1)}(\al)}{p_k(\al)+\bar{\lam}q_k^{(1)}(\al)}p_k(\az)1_{L_n}(\D z_k)dz_k\Bigg|^\mu\Bigg]\Bigg\}\\
    %     \nn&\intd \frac{|\p^5_\theta q_k^{(1)}(\check{\al}_n)|}{(p_k(\check{\al}_n)+\bar{\lam}q_k^{(1)}(\check{\al}_n))^{1-\frac{1}{2C^2}}}h_n^{-\frac{d}{2}+\frac{d}{4C^2}}\exp\Bigg(-\frac{|\D_k z|^2}{4C^2h_n}\Bigg)1_{L_n}(\D z_k)dz_k\\
    %     \nn&\lesssim \Bigg(\intd \Bigg(\frac{|\p^5_\theta q_k^{(1)}(\check{\al}_n)|}{q_k^{(1)}(\check{\al}_n)}\Bigg)^{2C^2}q_k^{(1)}(\check{\al}_n)1_{L_n}(\D z_k)dz_k\Bigg)^{\frac{1}{2C^2}}\\
    %     \nn&\lesssim h_n^{\frac{1}{2C^2}}(1+|z_{k-1}|^C).
    % \end{align}
    \begin{align}
    \nn&\Bigg|\intd \frac{\p^5_\theta q_k^{(1)}(\check{\al}_n)}{p_k(\check{\al}_n)+\bar{\lam}q_k^{(1)}(\check{\al}_n)}p_k(\az)1_{L_n}(\D z_k)dz_k\Bigg|\\
    \nn&\lesssim \intd \frac{|\p^5_\theta q_k^{(1)}(\check{\al}_n)|}{(p_k(\check{\al}_n)+\bar{\lam}q_k^{(1)}(\check{\al}_n))^{1-\frac{1}{C}}}h_n^{-\frac{d}{2}(1-\frac{1}{C})}\exp\Bigg(-\frac{|\D_k z|^2}{Ch_n}\Bigg)1_{L_n}(\D z_k)dz_k\\
        \nn&\lesssim \Bigg(\intd \Bigg(\frac{|\p^5_\theta q_k^{(1)}(\check{\al}_n)|}{q_k^{(1)}(\check{\al}_n)}\Bigg)^{C}q_k^{(1)}(\check{\al}_n)1_{L_n}(\D z_k)dz_k\Bigg)^{\frac{1}{C}}\Bigg(\intd h_n^{-\frac{d}{2}}\exp\Bigg(-\frac{|\D_k z|^2}{(C-1)h_n}\Bigg)dz_k\Bigg)^{1-\frac{1}{C}}\\
        \nn&\lesssim h_n^{\frac{1}{C}}(1+|z_{k-1}|^C).
    \end{align}
    By mimicking these calculations, it is easy to deduce \eqref{reminder5}.
    Moreover, for all $i\in\{0,1,2,3,4,5\}$, Assumption \ref{balance-cond} leads to
    $${\frac{n}{n^{\frac{i}{2}}(nh_n)^{\frac{5-i}{2}}}}=\sqrt{\frac{1}{n^3h_n^{5-i}}}\lesssim \sqrt{\frac{1}{n^{3-\frac{3}{5}(5-i)}}}\lesssim 1,$$
    and hence the moment convergence of the QMLE implies \eqref{reminder2}.
\end{proof}

Combined with Theorem \ref{ktesti}, and the fact that for all $p>0$,
$$\intd p_k(\az)1_{L_n^c}(\D z_k)dz_k=P(N_{\tz}((t_{k-1},t_k]\times\mbbr^d)=0 \ \text{and} \ |\D_k Z|> h_n^\rho|Z_{t_{k-1}})\lesssim h_n^p (1+|z_{k-1}|^C),$$
the argument of the proof of the previous lemma leads to
\begin{align*}
    &\intd \log \Bigg(1+\frac{p_k(\aes)}{q_k^{(1)}(\aes,\les)}\Bigg)q_k^{(1)}(\az,\lz)1_{L_n^c}(\D z_k)dz_k+\intd \log \Bigg(1+\frac{\bar{q}_k^{(2)}(\aes,\les)}{p_k(\aes)}\Bigg)p_k(\az)1_{L_n}(\D z_k)dz_k\\
    &+\intd \log \Bigg(1+\frac{\bar{q}_k^{(2)}(\aes,\les)}{q_k^{(1)}(\aes,\les)}\Bigg)q_k^{(1)}(\az,\lz)1_{L_n^c}(\D z_k)dz_k=o(1).
\end{align*}
Hence, we get \eqref{negli} and \eqref{app2}.

% \begin{align*}
%     &\Bigg|\int_{\mbbr^d} \log q_{k}(\al) (q_{k}(\al_0)-\tilde{q}_k(\alpha_0))dx_k\Bigg|\\
%     &\leq \int_{x_k\neq x_{k-1}} \log\left(C\left[\frac{1}{h_n^{d/2}}\exp\left(-\frac{|x_k-x_{k-1}|^2}{C'h_n}\right)+h_n\exp\left(-ue^{-1}||x_k-x_{k-1}||_1\right)\right]\right)(q_{k}(\al_0)-\tilde{q}_k(\al_0))dx_k\\
%     &\leq \left[C''+\log \left(h_n^{-d/2}+h_n\right)\right]\int_{x_k\neq x_{k-1}} (q_{k}(\al_0)-\tilde{q}_k(\al_0))dx_k,\\
%     &=o(n^{-1}).
% \end{align*}

\subsection{Proof of Theorem \ref{asmun}}
Let
$$\tilde{b}_{1,\lam}=\log \frac{\les}{\lz}\sumk\varphi_n(\D_k X)1_{L_n^c}(\D_k X)+nh_n\lz\int F_{\tz}(z)\varphi_n(z)dz-nh_n\les\int F_{\tes}(z)\varphi_n(z)dz.$$

\begin{Lem}\label{BQL}
    Under Assumptions \ref{inidist}-\ref{indentifiability}, we have
    \begin{equation}        
    E\Bigg[\mbbh_n(\aes,\les)-\mbbh_n(\az,\lz)-\tilde{b}_{1,\lam}\Bigg]=\frac{1}{2}\dim\Theta.
    \end{equation}
\end{Lem}
\begin{proof}
    Let $$\mbbh_{3,n}(\theta)=\sumk \log F_\theta(\D_k X)\varphi_n(\D_k X)1_{L_n^c}(\D_k X),$$
and $\bar{\mbbh}_n(\al)=\mbbh_{1,n}(\al)+\mbbh_{3,n}(\theta)$.
Then it follows that
$$\mbbh_n(\aes,\les)-\mbbh_n(\az,\lz)-\tilde{b}_{1,\lam}=\bmbbh_n(\aes)-\bmbbh_n(\az).$$
From Taylor expansion, we have
    \begin{align*}
        &\bmbbh_n(\aes)-\bmbbh_n(\az)\\
        &=\frac{1}{\sqrt{n}}\p_\sig\mbbh_{1,n}(\az)[\hat{u}_{n,\sig}]+\frac{1}{\sqrt{nh}}\p_\theta\bmbbh_n(\az)[\hat{u}_{n,\theta}]-\frac{1}{2}\Gam_\sig[\hat{u}_{n,\sig}^{\otimes2}]-\frac{1}{2}\Gam_\theta[\hat{u}_{n,\theta}^{\otimes2}]+R_{2,n}+R_{3,n},
    \end{align*}
    where $\hat{u}_{n,\sig}=\sqrt{n}(\hat{\sig}_n-\sig_0)$, $\hat{u}_{n,\theta}=\sqrt{nh_n}(\tes-\tz)$, and the reminder term $R_{2,n}$ and $R_{3,n}$ is defined by
    \begin{align*}
        R_{2,n}&=\frac{1}{2}\Big(\frac{1}{n}\p_\sig^2\mbbh_{1,n}(\az)+\Gam_\sig\Big)[\hat{u}_{n,\sig}^{\otimes2}]+\frac{1}{2}\Big(\frac{1}{nh_n}\p_\theta^2\bmbbh_n(\az)+\Gam_\theta\Big)[\hat{u}_{n,\theta}^{\otimes2}]\\
        &+\frac{1}{n\sqrt{h_n}}\p_\sig\p_\al\mbbh_{1,n}(\az)[\hat{u}_{n,\sig},\hat{u}_{n,\theta}],\\
        R_{3,n}&=\frac{1}{6}\p_\al^3\bmbbh_n(\check{\al}_n)[(\aes-\az)^{\otimes3}],
    \end{align*}
    with $\check{\al}_n=s\az+(1-s)\aes$ for a random $s\in(0,1)$.
    From the estimates in \cite[Lemma 4 and Lemma 6]{OgiYos11}, 
    % there exist positive constants $\beta_1$ and $\beta_2$ such that for all $p>p_\sig+p_\theta+1$,
    % \begin{equation}
    %     \nn E\Bigg[\Bigg|n^{\beta_1}\bigg(\frac{1}{n}\p_\sig^2\mbbh_{1,n}(\az)-\Gam_\sig\bigg)\Bigg|^p+\Bigg|(nh_n)^{\beta_2}\bigg(\frac{1}{nh_n}\p_\theta^2\bmbbh_n(\az)-\Gam_\theta\bigg)\Bigg|^p\Bigg]<\infty,
    % \end{equation}
    % By applying a similar estimate
    the moment convergence of the QMLE, Sobolev's inequality, and H\"{o}lder's inequality yield that
    $$E\big[\big|R_{2,n}\big|+\big|R_{3,n}\big|\big]=o(1).$$
    By applying Taylor expansion again, we get
    $$\hat{u}_{n,\sig}=\frac{1}{\sqrt{n}}\Gam_\sig^{-1}\p_\sig\mbbh_{1,n}(\az)+r_n,$$
    where $E[|r_n|]=o(1)$.
    Since $\frac{1}{\sqrt{n}}\p_\sig\mbbh_{1,n}(\az)\cil N(0,\Gam_\sig)$ and $E[|\frac{1}{\sqrt{n}}\p_\sig\mbbh_{1,n}(\az)|^q]<\infty$ for any $q>0$ from \cite{OgiYos11}, we obtain
    \begin{equation}
        \nn E\Bigg[\frac{1}{\sqrt{n}}\p_\sig\mbbh_{1,n}(\az)[\hat{u}_{n,\sig}]\Bigg]=\Tr\Bigg(\Gam^{-1}_\sig E\Bigg[\Bigg(\frac{1}{\sqrt{n}}\p_\sig\mbbh_{1,n}(\az)\Bigg)^{\otimes 2}\Bigg]\Bigg)+o(1)=\dim \Theta_\sig+o(1).
    \end{equation}
    In the same way, it follows that $E[\frac{1}{\sqrt{nh}}\p_\theta\bmbbh_n(\az)[\hat{u}_{n,\theta}]]=\dim\Theta_\theta+o(1)$.
    Finally, the moment convergence of the QMLE gives 
    $$E\Bigg[\frac{1}{2}\Gam_\sig[\hat{u}_{n,\sig}^{\otimes2}]+\frac{1}{2}\Gam_\theta[\hat{u}_{n,\theta}^{\otimes2}]\Bigg]=-\frac{1}{2}\dim\Theta+o(1),$$
    and hence the desired result follows.
\end{proof}

Let $q_k^{(1)}(\al)=\lam^{-1}q_{h_n}^{(1)}(z_{k-1},z_k;\al,\lam)$.
Since $\log \tilde{q}_k(\al,\lam)$ is decomposed as:
\begin{align*}
    \log \tilde{q}_k(\al,\lam)=(-\lam h_n+\log p_k(\al))1_{L_n}(\D z_k)+(-\lam h_n+\log \lam+\log q_{k}^{(1)}(\al))1_{L_n^c}(\D z_k),
\end{align*}
the estimates in the previous section and Proposition \ref{eKL} give
\begin{align*}
    &\Ez\Big[l_n(\mbfz_n,\aes,\les)\Big]-\Ez\Big[l_n(\mbfz_n,\az,\lz)\Big]-\tilde{b}_{2,\lam}=e^{-\lz h_n}(\mbbu_{1,n}+\lz\mbbu_{2,n})+o(1),
\end{align*}
where $\tilde{b}_{2,\lam}$, $\mbbu_{1,n}$, and $\mbbu_{2,n}$ are given by
\begin{align*}
    &\mbbu_{1,n}=\sumk\intd\intd \big(\log p_k(\aes)-\log p_k(\az)\big)p_k(\az)1_{L_n}(\D z_k)dz_k \pi_{k-1}(dz_{k-1}),\\
    &\mbbu_{2,n}=\sumk\intd\intd \Big(\log q_k^{(1)}(\aes)-\log q_k^{(1)}(\az)\Big) q_k^{(1)}(\az)1_{L_n^c}(\D z_k)dz_k \pi_{k-1}(dz_{k-1}),\\
    &\tilde{b}_{2,\lam}= e^{-\lz h_n}\log\frac{\les}{\lz}\sumk \int\int q_{h_n}^{(1)}(x,y;\az,\lz) 1_{L_n^c}(y-x)dy\pi_{k-1}(dx)+nh_n(\lz-\les).
\end{align*}

We will evaluate these random quantities separately below.

\begin{Lem}
    Under Assumptions \ref{inidist}-\ref{indentifiability}, it follows that
    $$E\Big[\tilde{b}_{1,\lam}-\tilde{b}_{2,\lam}\Big]=1+o(1).$$
\end{Lem}

\begin{proof}
    % From \cite[Proposition 4]{OgiYos11}, 
    % we have
    % $$\frac{1}{h_n}E\big[\varphi_n(\D_k X)1_{L_n^c}(\D_k X)|\mcf_{k-1}\big]=\lz\int \varphi_n(z)F_{\tz}(z)dz+Ch_n^{\frac{3}{8}}\big(1+|X_{t_{k-1
    % }}|^C\big)$$
    % From Taylor expansion, we have
    % \begin{align*}
    %     \int F_{\tes}(z)\varphi_n(z)dz=\int F_{\tz}(z)\varphi_n(z)dz+\frac{1}{\sqrt{nh_n}}\Bigg[\int \p_\theta F_{\check{\theta}_n}(z)\varphi_n(z)dz\Bigg][\hat{u}_{n,\theta}],
    % \end{align*}
    % with $\check{\theta}_n=s\tz+(1-s)\tes$ for a random $s\in(0,1)$.
    % Let $\hat{u}_{n,\lam}=\sqrt{nh_n}(\les-\lz)$.
    % Then, it follows from Assumption \ref{jsupp}-(3) and the moment convergence of the QMLE that
    % \begin{align*}
    %      &\Bigg|E\Bigg[\hat{u}_{n,\lam}\Bigg[\int \p_\theta F_{\check{\theta}_n}(z)\varphi_n(z)dz\Bigg][\hat{u}_{n,\theta}]\Bigg]\Bigg|\\
    %      &\leq E[|\hat{u}_{n,\lam}||\hat{u}_{n,\theta}]|\sup_\theta]
    % \end{align*}
    From Assumption \ref{jdist} and Assumption \ref{jsupp}, for all $l\in\{1,2,3\}$, an easy calculation leads to
    \begin{align*}
        \int |\p_\theta^l F_{\theta}(z)|\varphi_n(z)dz\leq C\int (1+|z|^C)e^{-C|z|}dz<\infty.
    \end{align*}
    Let $\hat{u}_{n,\lam}=\sqrt{nh_n}(\les-\lz)$.
    Then from a similar argument to the proof of Lemma \ref{BQL}, we get
    \begin{align}
        \nn E[\tilde{b}_{1,\lam}]&=E\Bigg[\frac{1}{\sqrt{nh_n}}\p_\lam \mbbh_{2,n}(\az,\lz)\hat{u}_{n,\lam}\Bigg]-\frac{1}{2\lz}E[\hat{u}_{n,\lam}^2]+o(1)\\
        \label{bl1}&=\frac{1}{\lz}E\Bigg[\Bigg(\frac{1}{\sqrt{nh_n}}\p_\lam \mbbh_{2,n}(\az,\lz)\Bigg)^2\Bigg]-\frac{1}{2}+o(1)=\frac{1}{2}.
    \end{align}
    By definition of $q_{h_n}^{(1)}(x,y;\az,\lz)$, we have
    \begin{align*}
        &\int\int q_{h_n}^{(1)}(x,y;\az,\lz) 1_{L_n^c}(y-x)dy\pi_{k-1}(dx)\\
        &=P(N_{\tz}((t_{k-1},t_k]\times\mbbr^d)=1)-\int P(N_{\tz}((t_{k-1},t_k]\times\mbbr^d)=1 \ \text{and} \ |\D_k Z|\leq h_n^\rho|Z_{t_{k-1}}=x)\pi_{k-1}(dx)\\
        &=\lz h_n-\int P(N_{\tz}((t_{k-1},t_k]\times\mbbr^d)=1 \ \text{and} \ |\D_k Z|\leq h_n^\rho|Z_{t_{k-1}}=x)\pi_{k-1}(dx).
    \end{align*}
    % \cite[Lemma 2]{OgiYos11} gives
    % \begin{align*}
    %     \int P(N_{\tz}((t_{k-1},t_k]\times\mbbr^d)=1 \ \text{and} \ |\D_k Z|\leq h_n^\rho|Z_{t_{k-1}}=x)\pi_{k-1}(dx)\lesssim h_n^{1+(d+\gam)\rho\wedge1},
    % \end{align*}
    Thus \eqref{pj1c}, Taylor expansion and the moment convergence of the QMLE imply that
    \begin{align}
        \nn E\big[\tilde{b}_{2,\lam}\big]&=-\frac{1}{2\lz}E\big[\hat{u}_{n,\lam}^2\big]+\sqrt{nh_n}(e^{-\lz h_n}-1)E[\hat{u}_{n,\lam}]+\frac{\lz}{3\check{\lam}_n^3}E\Big[\hat{u}_{n,\lam}^3\Big]+o(1)\\
        \label{bl2}&=-\frac{1}{2}+o(1),
    \end{align}
    with $\check{\lam}_n=s\lz+(1-s)\les$ for a random $s\in(0,1)$.
    In the last equality, we used the elementary inequality $|e^{-x}-1|\leq x$ for $x>0$.
    The desired result follows from \eqref{bl1} and \eqref{bl2}.
    % Decompose $\tilde{b}_{1,\lam}-\tilde{b}_{2,\lam}$ as:
    % \begin{align*}
    %     \tilde{b}_{1,\lam}-\tilde{b}_{2,\lam}=
    % \end{align*}
\end{proof}

Introduce $p_\theta\times p_\theta$-matrix $\Gam_\theta=\Gam_{\theta,1}+\Gam_{\theta_2}$ defined by
\begin{align*}
    [\Gamma_{\theta,1}]_{ij}=\int (\partial_{\theta_i} a)^\top S^{-1}(\partial_{\theta_j} a)(x,\alpha_0)d\pi_0(x), \ \ [\Gamma_{\theta,2}]_{ij}= \int \frac{\partial_{\theta_i} f_{\theta_0}\partial_{\theta_j} f_{\theta_0}}{f_{\theta_0}} 1_{\{f_{\theta_0}\neq 0\}}(z)dz.
\end{align*}

\begin{Lem}\label{Lestu1}
    Under Assumptions \ref{inidist}-\ref{indentifiability}, we have
    \begin{align}\label{estu1}
    E\bigg[\mbbu_{1,n}\bigg]=\frac{1}{2}p_\sig+\frac{1}{2}\Tr\Big(\Gam_{\theta,1}\Gam^{-1}_\theta\Big)+o(1).
    \end{align}
\end{Lem}

\begin{proof}
By Taylor expansion, we have
\begin{align*}
    &\mbbu_{1,n}\\
    &=\Bigg[\sumk\intd\intd \frac{\p_\al p_k(\az)}{p_k(\az)}p_k(\az)1_{L_n}(\D z_k)dz_k \pi_{k-1}(dz_{k-1})\Bigg]\big[\hat{v}_{n,\al}\big]\\
    &+\Bigg[\frac{1}{2n}\sumk\intd\intd \frac{\big(\p_\sig p_k(\az)\big)^{\otimes2}-p_k(\az)\p_\sig^2p_k(\az)}{p_k(\az)}1_{L_n}(\D z_k)dz_k\pi_{k-1}(dz_{k-1})\Bigg]\big[\hat{u}_{n,\sig}^{\otimes2}\big]\\
    &+\Bigg[\frac{1}{2nh_n}\sumk\intd\intd \frac{\big(\p_\theta p_k(\az)\big)^{\otimes2}-p_k(\az)\p_\theta^2p_k(\az)}{p_k(\az)}1_{L_n}(\D z_k)dz_k\pi_{k-1}(dz_{k-1})\Bigg]\big[\hat{u}_{n,\theta}^{\otimes2}\big]\\
    &+\Bigg[\frac{1}{n\sqrt{h_n}}\sumk\intd\intd \frac{\p_\theta p_k(\az)\p_\sig p_k(\az)-p_k(\az)\p_\theta\p_\sig p_k(\az)}{p_k(\az)}1_{L_n}(\D z_k)dz_k\pi_{k-1}(dz_{k-1})\Bigg]\big[\hat{u}_{n,\sig},\hat{u}_{n,\al}\big]\\
    &+\Bigg[\frac{1}{6}\sumk\intd\intd \p_\al^3 \bigg(\log p_k(\check{\al}_n)\bigg)p_k(\az)1_{L_n}(\D z_k)dz_k\pi_{k-1}(dz_{k-1})\Bigg]\big[(\hat{v}_n)^{\otimes3}\big],
\end{align*}
where $\check{\al}_n=s\az+(1-s)\aes$ for a random $s\in(0,1)$.
It follows from \cite[Lemma 2]{OgiYos11}, \cite[Proposition 5.2]{OgiUeh23} and Cauchy-Schwartz inequality that
\begin{align}
    &\nn\Bigg|\intd\intd \frac{\p_\al p_k(\az)}{p_k(\az)}p_k(\az)1_{L_n^c}(\D z_k)dz_k \pi_{k-1}(dz_{k-1})\Bigg|\\
    &\nn\leq \intd\sqrt{\intd\bigg|\frac{\p_\al p_k(\az)}{p_k(\az)}\bigg|^2p_k(\az)dz_k}\sqrt{\intd p_k(\az)1_{L_n^c}(\D z_k)dz_k}\pi_{k-1}(dz_{k-1})\\
    &\label{pest}\lesssim h_n^p,
\end{align}
for all $p\geq 2$.
We will often use such estimates on $L_n^c$ below.
Let $\del$ be the Hitsuda-Skorohod integral (the divergence operator).
From \cite[Proposition 4.1]{Gob01}, we have
\begin{align*}
    &\intd\intd \frac{\p_\al p_k(\az)}{p_k(\az)}p_k(\az)dz_k \pi_{k-1}(dz_{k-1})\\
    &=\intd\intd \frac{1}{h_n}\suml E\Bigg[\del\big(\p_{\al_{l_1}} Z^c_{h_n,l_1}\cdot U_{l_1}\big)\Bigg|Z^c_0=z_{k-1},Z^c_{h_n}=z_k\Bigg]p_k(\az)dz_k \pi_{k-1}(dz_{k-1})\\
    &= \intd\frac{1}{h_n}\suml E\bigg[\del\big(\p_{\al_{l_1}} Z^c_{h_n,l_1}\cdot U_{l_1}\big)\bigg|Z^c_0=z_{k-1}\bigg]\pi_{k-1}(dz_{k-1})\\
    &=0,
\end{align*}
and thus we get
\begin{equation}
    \Bigg|\sumk\intd\intd \frac{\p_\al p_k(\az)}{p_k(\az)}p_k(\az)1_{L_n}(\D z_k)dz_k \pi_{k-1}(dz_{k-1})\Bigg|=o(1).
\end{equation}
In a similar manner to \eqref{pest} and the proof of \cite[Lemma 2 (ii)]{Uch10}, we obtain
\begin{align*}
    &\Bigg|\frac{1}{n}\sumk\intd\intd \p_\sig^2p_k(\az)1_{L_n}(\D z_k)dz_k\pi_{k-1}(dz_{k-1})\Bigg|=o(1),\\
    &\Bigg|\frac{1}{nh_n}\sumk\intd\intd \p_\theta^2p_k(\az)1_{L_n}(\D z_k)dz_k\pi_{k-1}(dz_{k-1})\Bigg|=o(1),\\
    &\Bigg|\frac{1}{n\sqrt{h_n}}\sumk\intd\intd \p_\sig\p_\theta p_k(\az)1_{L_n}(\D z_k)dz_k\pi_{k-1}(dz_{k-1})\Bigg|=o(1).
\end{align*}
From \cite[Lemma D.1]{OgiUeh23}, there exist positive constant $C$ and positive decreasing sequence $\{\chi_n\}$ such that $\chi_n\to0$ and for all $z_{k-1}\in\mbbr^d$ and large enough $n$, 
\begin{align*}
    &\intd \Bigg|\frac{\p_\sig p_k(\az)}{p_k(\az)}+\frac{1}{2h_n} \p_\sig S^{-1}_{k-1}\Big[\big(\D z_k\big)^{\otimes2}\Big]-\frac{1}{2}\Tr\big(\p_\sig S^{-1}_{k-1}S_{k-1}\big)\Bigg|^21_{L_n}(\D z_k)p_k(\az)dz_k\\
    &\leq \chi_n\big(1+|z_{k-1}|^C\big),
    \end{align*}
    and
    $$\intd \Bigg|\frac{\p_\theta p_k(\az)}{p_k(\az)}-\p_\theta a^\top_{k-1}S^{-1}_{k-1}\D z_k\Bigg|^21_{L_n}(\D z_k)p_k(\az)dz_k\leq h_n\chi_n\big(1+|z_{k-1}|^C\big).$$
Combined with the arguments after \cite[Lemma D.2]{OgiUeh23} and Cauchy-Schwartz inequality, we obtain
\begin{align*}
    &\intd \frac{\big(\p_\sig p_k(\az)\big)^{\otimes2}}{p_k(\az)}1_{L_n}(\D z_k)dz_k\\
    &=E_{k-1}\Bigg[ \Bigg(\frac{1}{2h_n} \p_\sig S^{-1}_{k-1}\Big[\big(\D_k Z^c\big)^{\otimes2}\Big]-\frac{1}{2}\Tr\big(\p_\sig S^{-1}_{k-1}S_{k-1}\big)\Bigg)^{\otimes 2}\Bigg]+\chi_n\big(1+|z_{k-1}|^C\big)\\
    &=\bigg[\Tr\big(\p_{\sig^{(i)}} S_{k-1}^{-1}S_{k-1}\p_{\sig^{(j)}} S_{k-1}^{-1}S_{k-1}\big)\bigg]_{i,j}+\chi_n\big(1+|z_{k-1}|^C\big)
\end{align*}
for any $z_{k-1}\in\mbbr^d$.
Hence the ergodic theorem leads to 
\begin{align*}
    &\frac{1}{n}\sumk\intd\intd \frac{\big(\p_\sig p_k(\az)\big)^{\otimes2}}{p_k(\az)}1_{L_n}(\D z_k)dz_k\pi_{k-1}(dz_{k-1})\\
    &=\frac{1}{n}\sumk E\Bigg[\bigg[\Tr\big(\p_{\sig^{(i)}} S_{k-1}^{-1}S_{k-1}\p_{\sig^{(j)}} S_{k-1}^{-1}S_{k-1}\big)\bigg]_{i,j}\Bigg]+o(1)\\
    &\to \Gam_\sig.
\end{align*}
Similarly, for any $z_{k-1}\in\mbbr^d$, we get
\begin{align*}
    &\intd \frac{\big(\p_\theta p_k(\az)\big)^{\otimes2}}{p_k(\az)}1_{L_n}(\D z_k)dz_k\\
    &=E_{k-1}\Bigg[\Bigg(\p_\theta a^\top_{k-1}S^{-1}_{k-1}\D_k Z^c\Bigg)^{\otimes 2}\Bigg]+h_n\chi_n\big(1+|z_{k-1}|^C\big)\\
    &=h_n\p_\theta a^\top_{k-1}S^{-1}_{k-1}\p_\theta a_{k-1}+h_n\chi_n\big(1+|z_{k-1}|^C\big).
\end{align*}
where $Z^c$ is an independent copy of $X^c$.
Hence the ergodic theorem also leads to 
\begin{align*}
    &\frac{1}{nh_n}\sumk\intd\intd \frac{\big(\p_\theta p_k(\az)\big)^{\otimes2}}{p_k(\az)}1_{L_n}(\D z_k)dz_k\pi_{k-1}(dz_{k-1})\\
    &=\frac{1}{n}\sumk E\Bigg[\p_\theta a^\top_{k-1}S^{-1}_{k-1}\p_\theta a_{k-1}\Bigg]+o(1)\\
    &\to \Gam_{\theta,1}.
\end{align*}
In the same way, we obtain 
\begin{align*}
    \frac{1}{n\sqrt{h_n}}\sumk\intd\intd \frac{\p_\theta p_k(\az)\p_\sig p_k(\az)}{p_k(\az)}1_{L_n}(\D z_k)dz_k\pi_{k-1}(dz_{k-1})=o(1).
\end{align*}
Applying the estimates in \cite[Lemma 2 (iii)]{Uch10} and the moment convergence of the QMLE, we get \eqref{estu1}.
\end{proof}

\begin{Lem}\label{Lestu2}
    Under Assumptions \ref{inidist}-\ref{indentifiability}, we have
    \begin{align}\label{estu2}
    E\bigg[\mbbu_{2,n}\bigg]=\frac{1}{2\lz}\Tr\Big(\Gam_{\theta,2}\Gam_\theta^{-1}\Big)+o(1).
    \end{align}
\end{Lem}

\begin{proof}
From Taylor expansion, we have
\begin{align*}
    \mbbu_{2,n}&=\sumk\intd\intd \Big(\log q_k^{(1)}(\aes)-\log q_k^{(1)}(\az)\Big) q_k^{(1)}(\az)1_{L_n^c}(\D z_k)dz_k \pi_{k-1}(dz_{k-1})\\
    &=\Bigg[\sumk\intd\intd \p_\al q_k^{(1)}(\az)1_{L_n^c}(\D z_k)dz_k \pi_{k-1}(dz_{k-1})\Bigg]\big[\hat{v}_{n,\al}\big]\\
    &+\Bigg[\frac{1}{2}\sumk\intd\intd \frac{(\p_\al q_k^{(1)}(\az))^{\otimes2}-q_k^{(1)}(\az)\p_\al^2 q_k^{(1)}(\az)}{q_k^{(1)}(\az)}1_{L_n^c}(\D z_k)dz_k \pi_{k-1}(dz_{k-1})\Bigg]\big[(\hat{v}_{n,\al})^{\otimes 2}\big]\\
    &+\Bigg[\frac{1}{6}\sumk\intd\intd \p_\al^3\bigg(\log q_k^{(1)}(\check{\al}_n)\bigg)q_k^{(1)}(\az)1_{L_n^c}(\D z_k)dz_k \pi_{k-1}(dz_{k-1})\Bigg]\big[(\hat{v}_{n,\al})^{\otimes 3}\big],
\end{align*}
with $\check{\al}_n=s\az+(1-s)\aes$ for a random $s\in(0,1)$.

As was shown in the previous proof, for all $\tau\in(0,h_n)$ and $x\in\mbbr^d$, we have
\begin{align*}
    &\int \frac{\p_\al p_{\tau}(x,x_1;\al)}{p_{\tau}(x,x_1;\al)}p_{\tau}(x,x_1;\al)dx_1\\
    &=\int\frac{1}{\tau}\suml E\bigg[\del\big(\p_{\al_{l_1}} Z^c_{\tau,l_1}\cdot U_{l_1}\big)\bigg|Z^c_0=x,Z^c_{\tau}=x_1\bigg]p_{\tau}(x,x_1;\al)dx_1\\
    &=0.
\end{align*}
Moreover, it follows from Assumption \ref{jdist} that for all $\tau\in(0,h_n)$ and $x\in\mbbr^d$,
\begin{align*}
    &\intd\int\int p_{\tau}(x,x_1;\al) \p_\theta F_\theta(x_2) p_{h_n-\tau}(x_1+x_2, y; \al) dx_1dx_2dy\\
    &=\int \p_\theta F_\theta(x_2)dx_2=\p_\theta\Bigg(\int F_\theta(x_2)dx_2\Bigg)=0.
\end{align*}
Hence we obtain
    \begin{align}\label{q10}
        \sumk\intd\intd \p_\al q_k^{(1)}(\az) dz_k \pi_{k-1}(dz_{k-1})=0,
    \end{align}
and combined with \eqref{dpj1c}, we do
    \begin{align}\label{qc10}
        \sumk\intd\intd \p_\al q_k^{(1)}(\az) 1_{L_n^c}(\D z_k)dz_k \pi_{k-1}(dz_{k-1})=o(1).
    \end{align}
In the same way, we also have
$$\sumk\intd\intd \p_\al^2 q_k^{(1)}(\az) 1_{L_n^c}(\D z_k)dz_k \pi_{k-1}(dz_{k-1})=o(1).
$$
By \cite[Proposition 5.4]{OgiUeh23}, there exists a positive constant $\ep\in(0,1)$ such that
\begin{align*}
    \intd \Bigg|\frac{\p_\theta q_k^{(1)}(\az)}{q_k^{(1)}(\az)}-\frac{\p_\theta F_{\tz}(\D z_k)}{F_{\tz}(\D z_k)}\Bigg|^2q_k^{(1)}(\az)1_{L_n^c}(\D z_k)dz_k\lesssim h_n^{1+\ep}(1+|z_{k-1}|^C).
\end{align*}
Together with Lemma \ref{derpar} and \cite[Lemma D.3]{OgiUeh23}, it follows that
\begin{align*}
    &E\Bigg[\Bigg[\sumk\intd\intd \frac{(\p_\al q_k^{(1)}(\az))^{\otimes2}-q_k^{(1)}(\az)\p_\al^2 q_k^{(1)}(\az)}{q_k^{(1)}(\az)}1_{L_n^c}(\D z_k)dz_k \pi_{k-1}(dz_{k-1})\Bigg]\big[(\hat{v}_{n,\al})^{\otimes 2}\big]\Bigg]\\
    &=\Bigg[\frac{1}{nh_n}\sumk \intd\intd \frac{(\p_\theta F_{\tz}(\D z_k))^{\otimes 2}}{F_{\tz}(\D z_k)}q_k^{(1)}(\az)1_{L_n^c}(\D z_k)dz_k\Bigg]E[(\hat{u}_{n,\theta})^{\otimes 2}]+o(1)\\
    &=\frac{1}{\lz}\Tr(\Gam_{\theta,2}\Gam_\theta^{-1})+o(1).
\end{align*}
Finally, from a similar estimate used in the proof of Lemma \ref{lem:negli}, we have
$$E\Bigg[\Bigg[\sumk\intd\intd \p_\al^3\bigg(\log q_k^{(1)}(\check{\al}_n)\bigg)q_k^{(1)}(\az)1_{L_n^c}(\D z_k)dz_k \pi_{k-1}(dz_{k-1})\Bigg]\big[(\hat{v}_{n,\al})^{\otimes 3}\big]\Bigg]=o(1),$$
and thus the desired result follows.
\end{proof}

By combining Proposition \ref{eKL} and Lemmas \ref{BQL}-\ref{Lestu2}, we get Theorem \ref{asmun}.

\subsection{Proof of Theorem \ref{relpro}}
Although the proof is almost the same as \cite[Theorem 4.4 (c)]{FasKim17} and \cite[Theorem 5]{EguMas24}, we give a self-contained proof to improve readability.
Define a map $f:\mbbr^{p_{\mcm_\star}}\to\mbbr^{p_{\mcm}}$ by 
$$f(x)=Fx+c.$$
From the additional nested condition and the estimates of $\tilde{b}_{1,\lam}$, Taylor expansion yields that
\begin{align*}
    &\mbbh_n^{(\mcm)}(\aes,\les)-\mbbh_n^{(\mcm_\star)}(\aes^{(\star)},\les)\\
    &=\bar{\mbbh}_{n}^{(\mcm)}(\aes)-\bar{\mbbh}_{n}^{(\mcm_\star)}(\aes^{(\star)})+o_p(1)\\
    &=-\p_\al \bar{\mbbh}_{n}^{(\mcm)}(\aes)[f(\aes^{(\star)})-\aes]-\frac{1}{2}\p_\al^{ 2} \bar{\mbbh}_{n}^{(\mcm)}(\check{\al}_n)\big[(f(\aes^{(\star)})-\aes)^{\otimes2}\big]+o_p(1),\\
    &=-\frac{1}{2}\p_\al^{ 2} \bar{\mbbh}_{n}^{(\mcm)}(\check{\al}_n)\big[(f(\aes^{(\star)})-\aes)^{\otimes2}\big]+o_p(1)
\end{align*}
with $\check{\al}_n=s\aes+(1-s)f(\aes^{(\star)})$ for a random $s\in(0,1)$.
% Let $A_n=\diag\{\sqrt{nh_n}I_{p_\theta},\sqrt{n}I_\sig\}$ and $A_n^\star=\diag\{\sqrt{nh_n}I_{p_{\theta_\star}},\sqrt{n}I_{\sig_\star}\}$.
Then, from $f(\al_0^{(\star)})=\al_0$, the chain rule leads to $\p_{\sig^{(\star)}}\bar{\mbbh}_n^{\mcm_\star}(\al_0^{(\star)})=F^\top_\sig\p_{\sig}\bar{\mbbh}_n^{\mcm}(\al_0)$ and $\p_{\sig^{(\star)}}^2\bar{\mbbh}_n^{\mcm_\star}(\al_0^{(\star)})=F^\top_\sig\p_{\sig}^2\bar{\mbbh}_n^{\mcm}(\al_0)F_\sig$.
Hence, Taylor expansion and the estimates of the quasi-likelihood function yield
\begin{align*}
    &\sqrt{n}(f_\sig(\hat{\sig}_n^{(\star)})-\hat{\sig}_n)\\
    &=\sqrt{n}(f_\sig(\hat{\sig}_n^{(\star)})-f_\sig(\sig_0^{(\star)}))-\sqrt{n}(\hat{\sig}_n-\sig_0)\\
    &=F_\sig\Bigg(F_\sig^\top\frac{1}{n}\p_{\sig}^2\bar{\mbbh}_n^{\mcm}(\al_0)F_\sig\Bigg)^{-1}F_\sig^\top\frac{1}{\sqrt{n}}\p_{\sig}\bar{\mbbh}_n^{\mcm}(\al_0)-\Bigg(\frac{1}{n}\p_{\sig}^2\bar{\mbbh}_n^{\mcm}(\al_0)\Bigg)^{-1}\frac{1}{\sqrt{n}}\p_{\sig}\bar{\mbbh}_n^{\mcm}(\al_0)+o_p(1)\\
    &=F_\sig\Big(\Gam^{-1/2}_\sig-\big(F_\sig^\top\Gam_\sig F_\sig\big)^{-1}F_\sig^\top \Gam^{1/2}_\sig\Big)\mcn_\sig+o_p(1),
\end{align*}
where the map $f_\sig:\mbbr^{p_\sig^{(\star)}}\to\mbbr^{p_\sig}$ is defined as $f_\sig(x)=F_\sig x+c_\sig$ and $\mcn_\sig\sim N(0, I_{p_\sig})$.
Clearly, a similar decomposition holds for $\sqrt{nh_n}(f(\hat{\theta}_n^{(\star)})-\hat{\theta}_n)$.
Let 
$$\Gam_{F,\sig,\theta}=\diag\{\Gam^{1/2}_\theta F_\theta(F_\theta^\top\Gam_\theta F_\theta)^{-1}F_\theta^\top\Gam^{1/2}_\theta,\Gam^{1/2}_\sig F_\sig(F_\sig^\top\Gam_\sig F_\sig)^{-1}F_\sig^\top\Gam^{1/2}_\sig\}.$$
It is easy to see that $I_{p_\mcm}-\Gam_{F,\sig,\theta}$ is a projection matrix and thus its eigenvalue is $0$ or $1$.
Moreover, since we have
\begin{align*}
    &\Tr(I_{p_\mcm}-\Gam_{F,\sig,\theta})\\
    &=\Tr(I_{p_\mcm})-\Tr(\diag\{\Gam^{1/2}_\theta F_\theta(F_\theta^\top\Gam_\theta F_\theta)^{-1}F_\theta^\top\Gam^{1/2}_\theta,\Gam^{1/2}_\sig F_\sig(F_\sig^\top\Gam_\sig F_\sig)^{-1}F_\sig^\top\Gam^{1/2}_\sig\})\\
    &=\Tr(I_{p_\mcm})-\Tr(I_{p_{\mcm_\star}})=p_\mcm-p_{\mcm_{\star}},
\end{align*}
there exists an orthogonal matrix $U$ such that
$$U^\top (I_{p_\mcm}-\Gam_{F,\sig,\theta})U=
\begin{blockarray}{ccccccc}
\BAmulticolumn{3}{c}{\overbrace{\hspace{5em}}^{p_\mcm-p_{\mcm_{\star}}}} &\BAmulticolumn{3}{c}{\overbrace{\hspace{5em}}^{p_{\mcm_{\star}}}}\\
\begin{block}{(ccccccc)}
1&&&&&&\\
&\ddots&&&&\\
&&1&&&\\
&&&0&&\\
&&&&\ddots&\\
&&&&&0\\
\end{block}
\end{blockarray}
$$
Then, by letting $\mcn\sim N(0, I_{p_\mcm})$ and $\tilde{\mcn}=U^\top\mcn$, we obtain
\begin{align*}
    &2\Big(\mbbh_n^{(\mcm)}(\aes,\les)-\mbbh_n^{(\mcm_\star)}(\aes^{(\star)},\les)\Big)\\
    &=-\p_\al^{ 2} \bar{\mbbh}_{n}^{(\mcm)}(\check{\al}_n)\big[(f(\aes^{(\star)})-\aes)^{\otimes2}\big]+o_p(1)
    % &=\mcn^\top \Big(\Gam^{-1/2}-\Gam^{1/2}F\big(F^\top\Gam F\big)^{-1}\Big)\Gam\Big(\Gam^{-1/2}-\big(F^\top\Gam F\big)^{-1}F^\top \Gam^{1/2}\Big)\mcn+o_p(1)\\
    =\mcn^\top(I_{p_\mcm}-\Gam_{F,\sig,\theta})\mcn+o_p(1)\\
    &=\sum_{j=1}^{p_\mcm-p_{\mcm_{\star}}} \tilde{\mcn}_j^2+o_p(1),
\end{align*}
Hence, the proof is complete.

\section*{Acknowledgment}
This work was supported by JSPS KAKENHI Grant Number JP19K20230.

\bibliographystyle{abbrv} 
\bibliography{YU_bibs.bib}

\end{document}